\documentclass[reqno,a4paper,11pt,oneside,final]{amsart}
\usepackage{etex}

\usepackage{fixltx2e}
\usepackage{float}
\usepackage{mathrsfs}
\usepackage{amssymb,amsmath,amsthm,amsfonts}
\usepackage{mathtools}
\usepackage{fixmath}
\usepackage{bbm}
\usepackage{braket}
\usepackage{caption}
\usepackage{subcaption}
\usepackage{enumitem}
\usepackage{units}
\usepackage[final]{graphicx}
\usepackage{xcolor}
\usepackage{booktabs}
\usepackage[english]{babel}
\usepackage[T1]{fontenc}
\usepackage{lmodern}
\usepackage{tikz}
\usepackage{pgfplots}
\usepackage{algorithm}
\usepackage{algorithmic}
\usepackage[final,stretch=10]{microtype}
\usepackage[colorlinks, citecolor=hanblue,linkcolor=red]{hyperref}
\definecolor{hanblue}{rgb}{0.27, 0.42, 0.81}
\usepackage{todonotes}

\DeclarePairedDelimiter{\norm}{\lVert}{\rVert}
\DeclarePairedDelimiter{\pair}{\langle}{\rangle}

\DeclareMathOperator{\dom}{\operatorname{dom}}

\DeclareMathOperator*{\argmin}{arg\,min}

\newcommand{\N}{\mathbb{N}}

\newcommand{\R}{\mathbb{R}}

\newcommand{\Pb}{\mbox{\rm (P) }}

\newcommand{\optu}{\bar{u}}

\newtheorem{mydef}{Definition}

\newcommand{\cnorm}[1]{\|#1\|_{\mathcal{C}}}

\newcommand{\tr}[1]{tr(C(1)^{-1})}

\newcommand{\eps}{\varepsilon}

\newcommand{\vertiii}[1]{{\left\vert\kern-0.25ex\left\vert\kern-0.25ex\left\vert #1
    \right\vert\kern-0.25ex\right\vert\kern-0.25ex\right\vert}}

\newcommand{\unorm}[1]{\norm{#1}_{\mathcal{U}}}
\newcommand{\stnorm}[1]{\norm{#1}_{*}}

\makeatletter
\DeclareFontFamily{OMX}{MnSymbolE}{}
\DeclareSymbolFont{MnLargeSymbols}{OMX}{MnSymbolE}{m}{n}
\SetSymbolFont{MnLargeSymbols}{bold}{OMX}{MnSymbolE}{b}{n}
\DeclareFontShape{OMX}{MnSymbolE}{m}{n}{
    <-6>  MnSymbolE5
   <6-7>  MnSymbolE6
   <7-8>  MnSymbolE7
   <8-9>  MnSymbolE8
   <9-10> MnSymbolE9
  <10-12> MnSymbolE10
  <12->   MnSymbolE12
}{}
\DeclareFontShape{OMX}{MnSymbolE}{b}{n}{
    <-6>  MnSymbolE-Bold5
   <6-7>  MnSymbolE-Bold6
   <7-8>  MnSymbolE-Bold7
   <8-9>  MnSymbolE-Bold8
   <9-10> MnSymbolE-Bold9
  <10-12> MnSymbolE-Bold10
  <12->   MnSymbolE-Bold12
}{}

\let\llangle\@undefined
\let\rrangle\@undefined
\DeclareMathDelimiter{\llangle}{\mathopen}%
                     {MnLargeSymbols}{'164}{MnLargeSymbols}{'164}
\DeclareMathDelimiter{\rrangle}{\mathclose}%
                     {MnLargeSymbols}{'171}{MnLargeSymbols}{'171}
\makeatother

\makeatletter
\newcommand*\rel@kern[1]{\kern#1\dimexpr\macc@kerna}
\newcommand*\widebar[1]{%
  \begingroup
  \def\mathaccent##1##2{%
    \rel@kern{0.8}%
    \overline{\rel@kern{-0.8}\macc@nucleus\rel@kern{0.2}}%
    \rel@kern{-0.2}%
  }%
  \macc@depth\@ne
  \let\math@bgroup\@empty \let\math@egroup\macc@set@skewchar
  \mathsurround\z@ \frozen@everymath{\mathgroup\macc@group\relax}%
  \macc@set@skewchar\relax
  \let\mathaccentV\macc@nested@a
  \macc@nested@a\relax111{#1}%
  \endgroup
}
\makeatother

\numberwithin{equation}{section}

\definecolor{darkred}{rgb}{.7,0,0}

\definecolor{green}{rgb}{0,0.7,0}

\usepackage[notref,notcite]{showkeys}
\theoremstyle{plain}

\newtheorem{theorem}{Theorem}[section]
\newtheorem{lemma}[theorem]{Lemma}

\newtheorem{proposition}[theorem]{Proposition}

\theoremstyle{definition}

\newtheorem{assumption}[theorem]{Assumption}

\newtheorem{example}[theorem]{Example}
\newtheorem{remark}[theorem]{Remark}

\begin{document}
\title[Fast generalized conditional gradient]{On fast convergence rates for generalized conditional gradient methods with backtracking stepsize}

\pagestyle{myheadings}

 \author[K. Kunisch]
 {Karl Kunisch}
 \address[Karl Kunisch]{University of Graz, Institute of Mathematics and Scientific Computing, Heinrichstra\ss e 36, 8010 Graz, Austria}
\email{karl.kunisch@uni-graz.at}

\author[D. Walter]{Daniel Walter}
\address[Daniel Walter]{Johann Radon Institute for Compuational and Applied Mathematics, Altenberger Stra\ss e 69, 4040 Linz}
 \email{daniel.walter@oeaw.ac.at}
 \thanks{Karl Kunisch was supported by the ERC advanced grant 668998 (OCLOC) under the EU’s
H2020 research program.}

\begin{abstract}
A generalized conditional gradient method for minimizing the sum of two convex functions, one of them differentiable, is presented. This iterative method relies on two main ingredients: First, the minimization of a partially linearized objective functional to compute a descent direction and, second, a stepsize choice based on an Armijo-like condition to ensure sufficient descent in every iteration. We provide several convergence results. Under mild assumptions, the method generates sequences of iterates which converge, on subsequences, towards minimizers. Moreover a sublinear rate of convergence for the objective functional values is derived. Second, we show that the method enjoys improved rates of convergence if the partially linearized problem fulfills certain growth estimates. Most notably these results do not require strong convexity of the objective functional. Numerical tests on a variety of challenging PDE-constrained optimization problems confirm the practical efficiency of the proposed algorithm.
\end{abstract}
\subjclass{	90C25, 49J52,65K10 , 49M41 }
\keywords{optimization methods, generalized conditional gradient, nonsmooth optimization}

%\vfil
%\hfil Recent Advances in Data-Driven Inverse Problems and Learning \hfil
%\vfil
%\newpage
\maketitle

\section{Introduction} \label{sec:intro}
The conditional gradient method,~\cite{frank}, provides a simple tool for the solution of constrained minimization problems of the form
\begin{align}  \label{def:constprob}
\min_{u \in \mathcal{U}_{ad}} f(u) \tag{$\mathcal{P}_{\mathcal{U}_{ad}}$}
\end{align}
where~$\mathcal{U}_{ad}$ is a convex, compact subset of a Banach space~$\mathcal{U}$ and~$f$ is a smooth convex function. This  algorithm generates a sequence of iterates~$u^k$ by minimizing the linearization of~$f$ to obtain a descent direction and updating the iterate by forming a convex combination:
\begin{align} \label{CGvanilla}
v^k \in \argmin_{v \in \mathcal{U}_{ad}} \langle \nabla f(u^k),v \rangle,~u^{k+1}=u^k+s^k(v^k-u^k),~s^k\in [0,1].
\end{align}
While it is well-known that this method admits a~$\mathcal{O}(1/k)$ rate of convergence,~\cite{jaggi}, it comes with two desirable properties which outweigh its poor theoretical performance. First, every iteration only requires the solution of a constrained linear minimization problem. This is often easier or computationally less demanding than e.g. the calculation of a projection. We refer to~\cite{pokuttacomplex} for some finite dimensional examples. Second, the descent direction~$v^k$ can always be chosen as an extremal point of~$\mathcal{U}_{ad}$. Depending on the particular constraint set~$\mathcal{U}_{ad}$, this might lead to iterates~$u^k$ exhibiting certain structural features such as sparsity or low rank in the early iterations. These advantages make algorithms of the form~\eqref{CGvanilla} attractive in a large variety of interesting problems. For examples we point  to~\cite{kerdreux,jaggi} and the references therein.

While the sublinear rate of this method is known to be tight in general,~\cite{cannon}, improved convergence results have already been derived in a variety of settings. For example, if~$f$ and~$\mathcal{U}_{ad}$ are strongly convex and the unique minimizer to~\eqref{def:constprob} lies in the interior of~$\mathcal{U}_{ad}$, then the algorithm converges linearly. The same holds true, in the strongly convex case, if~$\nabla f$ is bounded away from~$0$ on~$\mathcal{U}_{ad}$,~\cite{levitin, demyanov}. Moreover~\cite{garber} provides a~$\mathcal{O}(1/k^2)$ rate of convergence in the strongly convex case~\textit{without} further structural assumptions.
The results in~\cite{levitin, demyanov} were later generalized by~\cite{Dunn79, dunnimplicit, kerdreux} which show that improved convergence rates can be obtained without strong convexity assumptions if growth conditions of the form
\begin{align*}
 \langle \nabla f(\bar{u}),u-\optu \rangle \geq \theta \norm{u-\optu}^q \quad \forall u \in \mathcal{U}_{ad}
\end{align*}
hold for some~$\theta>0$,~$q \geq 2$.

More recently, there has been an increased interest in generalizing conditional gradient methods to problems of the form
\begin{align}
\label{def:problem}
\min_{u \in \mathcal{U}} j(u) \coloneqq \left \lbrack f(u)+g(u) \right \rbrack\tag{$\mathcal{P}$}
\end{align}
where~$f$ is a smooth, convex function and~$g$ is convex but possibly nonsmooth. Problems of this form often appear e.g. in the context of optimal control and inverse problems with~$g$ representing a (nonsmooth) cost or regularization term which favours minimizers with certain structural properties. In this setting, the \textit{generalized conditional gradient} (GCG) method, possibly first stated in~\cite{bredies}, proceeds as follows:
\begin{align} \label{GCGvanilla}
v^k \in \argmin_{v \in \mathcal{U} } \left\lbrack \langle \nabla f(u^k),v \rangle +g(v) \right \rbrack ,~u^{k+1}=u^k+s^k(v^k-u^k),~s^k\in [0,1].
\end{align}
Here, the smooth part~$f$ of the objective functional is linearized around~$u^k$ while the nonsmooth part~$g$ remains unchanged. Variants of this method have been applied to great success in applications such as super-resolution,~\cite{denoyelle}, acoustic inversion,~\cite{tang}, and dynamic transport regularization,~\cite{fanzon}, which all go beyond the setting of~\eqref{def:constprob}.
Of course, if~$g=I_{\mathcal{U}_{ad}}$, the convex indicator function of~$\mathcal{U}_{ad}$,~\eqref{def:problem} reduces to~\eqref{def:constprob} and the original conditional gradient method is recovered. Sublinear rates of convergence for the GCG method have already been proven for a variety of stepsize choices. Without any pretense of completeness we refer to~\cite{bredies, yu,xu2017,rakotomamonjy} as well as~\cite[Chapter 13]{beckbook}.
Nevertheless literature on improved convergence results for~$g \neq I_{\mathcal{U}_{ad}} $ is scarce. We point to ~\cite{flinth,denoyelle,pieper,trautmann} for examples of~\textit{accelerated} GCG methods, i.e.~\eqref{GCGvanilla} with additional speed-up steps, which achieve linear, or even finite-step, convergence in particular settings. Faster convergence results \textit{without} augmentations to~\eqref{GCGvanilla}, however, seem not to be available.

In the present manuscript a first step is taken at closing this gap. More concretely, our contributions are as follows: We propose a GCG method in the spirit of~\eqref{GCGvanilla} where~$s^k$ is chosen based on an Armijo-like backtracking condition. This choice of the stepsize offers a good trade-off between computational effort and the necessary per-iteration descent to achieve convergence. As far as we know, GCG with backtracking stepsizes have only been briefly mentioned in~\cite[Remark 2]{bredies} and utilized in one of the author's earlier papers,~\cite{neitzel}, for one particular instance of Problem~\eqref{def:problem}. For this reason we start by proving the convergence of the method as well as a sublinear~$\mathcal{O}(1/k)$ rate of convergence for the objective functional values, see Theorem~\ref{thm:slowconvergence}. Subsequently, in Theorems~\ref{thm:fastconvergence1} and~\ref{thm:fastconvergence2}, we show that the method enjoys improved convergence rates if the partially linearized objective functional fulfills a growth condition of the form
\begin{align} \label{growthGCG}
\langle \nabla f(\bar{u}),u-\optu \rangle+g(u)-g(\optu) \geq \theta \norm{u-\optu}^q \quad \forall u \in \mathcal{U}_{ad}
\end{align}
around a minimizer~$\optu$. For~$q=2$, we obtain a linear~$\mathcal{O}(\lambda^k)$,~$0<\lambda <1$, rate of convergence, while the objective functional values converge at least with rate~$\mathcal{O}(1/k^{1/\beta})$,~$\beta=1-2/(q(q-1))$ for~$q>2$. This generalizes the results of~\cite{dunnimplicit,kerdreux} to the GCG method.
To the best of our knowledge these are the first convergence results for \eqref{GCGvanilla} which go~\textit{beyond} a sublinear convergence rate. Finally, we discuss the application of our algorithm to challenging examples from the literature, such as bang-bang-off or directional sparsity problems, and discuss the growth condition~\eqref{growthGCG} in this context. Numerical examples highlight the practical efficiency of the method and confirm our theoretical results.

The rest of the manuscript is structured as follows: In Section~\ref{sec:setting}, we argue the existence of minimizers to~\eqref{def:problem} under appropriate assumptions on the involved functionals and derive first order optimaliy conditions. Subsequently, in Section~\ref{sec:generalizedfrankwolfe}, the GCG method is presented and its individual steps are examined in detail. Finally, Section~\ref{sec:convergence} addresses the convergence behaviour of the algorithm. The paper is finished by Section~\ref{sec:applications} in which we discuss and showcase the practical realization of GCG methods for several examples from PDE-constrained optimization.
\section{Problem setting and optimality conditions} \label{sec:setting}
Throughout the paper we assume that there is a separable Banach space~$\mathcal{C}$ which is the topological predual space of~$\mathcal{U}$, i.e.~$\mathcal{C}^* \simeq \mathcal{U}$. The corresponding duality pairing is denoted by~$\langle \cdot,\cdot \rangle$. This makes~$\mathcal{U}$ a Banach space with respect to the canonical norm
\begin{align*}
\unorm{u}= \sup_{\cnorm{\varphi}\leq 1} \langle \varphi,u \rangle  \quad \text{for all } u \in \mathcal{U}.
\end{align*}
A sequence~$u^k \in \mathcal{U}$,~$k\in\N$, is called weak* convergent with limit~$\optu$,~denoted by~$u^k \rightharpoonup^* \optu$, if
\begin{align*}
\langle \varphi, u^k \rangle \rightarrow \langle \varphi,\optu \rangle \quad \forall \varphi \in \mathcal{C}.
\end{align*}
Moreover the following assumptions concerning Problem~\eqref{def:problem} are made.
\begin{assumption} \label{ass:functions}\phantom{ }

\begin{itemize}
\item[$\mathbf{A1}$.] The function~$f \colon \mathcal{U} \to \mathbb{R} $ is convex, weak* continuous and G\^{a}teaux-differentiable. Moreover there exists a weak*-to-strong continuous mapping~$\nabla f \colon \mathcal{U} \to \mathcal{C}$ with
\begin{align*}
f'(u)(\delta u)= \langle \nabla f(u),\delta u \rangle \quad \forall u, \delta u \in \mathcal{U}.
\end{align*}
\item[$\mathbf{A2}$.] The function~$g \colon \mathcal{U} \to \bar{\mathbb{R}}_+ $ is proper, convex, weak* lower semicontinuous, and coercive, i.e.
\begin{align*}
\unorm{u} \rightarrow +\infty \Rightarrow g(u)/\unorm{u} \rightarrow +\infty.
\end{align*}
\item[$\mathbf{A3}$.] The functional~$j= f+g$ is radially unbounded and satisfies~$\inf_{u \in \mathcal{U}} j(u)>-\infty$.
\end{itemize}
\end{assumption}
It is also convenient to introduce for each~$u\in \mathcal{U}$ the sublevel set
\begin{align} \label{def:sublevelset}
E_j(u)=\left\{\,v \in \mathcal{U}\;|\;j(v)\leq j(u)\,\right\}.
\end{align}
Using Assumption~\ref{ass:functions}, we see that~$E_j(u)$ is weak* sequentially compact if~$j(u)<+\infty$.
This also guarantees the existence of at least one minimizer to~\eqref{def:problem}.
\begin{proposition}
\label{prop:existence}
Let Assumption~\ref{ass:functions} hold. Then the set~$E_j(u)$ is weak* (sequentially) compact for every~$u\in \dom j$. Moreover there exists at least one minimizer~$\optu$ to~\eqref{def:problem}.
\end{proposition}
\begin{proof}
First observe that~$j$ is weak* lower semicontinuous, see Assumption~\ref{ass:functions}~$\mathbf{A1}$ and~$\mathbf{A2}$. Now let~$v_k \in E_j(u)$,~$k\in\N$, be an arbitrary sequence.
Since~$j$ is radially unbounded, Assumption~\ref{ass:functions}~$\mathbf{A3}$,~$E_j(u)$, and thus also~$v_k$, is uniformly bounded in $\mathcal{U}$. Hence, due to the sequential Banach-Alaoglu theorem, there is~$\bar{v}\in \mathcal{U}$ and a subsequence of~$v_k$,~$k\in\N$, denotes by the same subscript with~$v_k \rightharpoonup^* \bar{v}$. The weak* lower semicontinuity of~$j$ now yields
\begin{align*}
j(\bar{v}) \leq \liminf_{k \rightarrow \infty} j(v_k) \leq j(u)
\end{align*}
and thus~$\bar{v} \in E_j(u)$. Consequently~$E_j(u)$ is weak* sequentially compact. The existence of a minimizer to~\eqref{def:problem} now follows by the standard method of calculus of variations.
\end{proof}
In addition to Assumption~\ref{ass:functions} we will require the following Lipschitz-like condition on~$\nabla f$  throughout the paper.
\begin{assumption} \label{ass:lipschitz}
There is a norm~$\norm{\cdot}_*$ on~$\mathcal{U}$ such that:
\begin{itemize}
\item[\textbf{A4}] Every sequence~$u_k \in \mathcal{U}$,~$k\in\N$, which is bounded w.r.t.~$\unorm{\cdot}$ is also bounded w.r.t.~$\stnorm{\cdot}$.
\item[\textbf{A5}] For every~$u \in \mathcal{U}$ there exists~$L_u>0$ only depending on~$E_j(u)$ such that we have
\begin{align*}
\langle \nabla f(u_1)-\nabla f(u_2), u_3-u_4 \rangle \leq L_u \stnorm{u_1-u_2} \stnorm{u_3-u_4}
\end{align*}
for all~$u_1,u_2 \in E_j(u),~u_3,u_4 \in \dom j$.
\end{itemize}
\end{assumption}
In the analysis of~\eqref{def:problem}, a major role is played by its first order necessary and sufficient optimality conditions which, see e.g.~\cite[Section 27]{rockafellar}, are given by the variational inequality
\begin{align} \label{eq:subdiff}
\langle \nabla f(\optu), \optu-u \rangle+ g(\optu) \leq g(u) \quad \forall u \in \mathcal{U}.
\end{align}
For our purposes it will be beneficial to make a connection between~\eqref{eq:subdiff} and the~\textit{gap functional}
\begin{align} \label{def:gap}
\Psi(u)= \sup_{v \in \mathcal{U}}\left \lbrack \langle \nabla f(u),u-v \rangle+g(u)-g(v) \right\rbrack \geq 0.
\end{align}
It turns out that its zeros are precisely the minimizers of~\eqref{def:problem}. Moreover~$\Psi(u)$ gives an upper bound on the suboptimality of~$u \in \mathcal{U} $.
\begin{theorem}
\label{thm:optimality}
Let~$\bar{u} \in \mathcal{U}$ be given. Then~$\optu$ satisfies the first order optimality condition~\eqref{eq:subdiff} if and only if~$\Psi(\optu)=0$. Additionally, for every~$u \in \dom j $ we have
\begin{align} \label{eq:boundforres}
j(u)-\min_{u \in \mathcal{U}} j(u) \leq \Psi(u),
\end{align}
and~$\Psi$ is weak* lower semicontinuous.
\end{theorem}
\begin{proof}
Assume that~$\bar{u}$ fulfills~\eqref{eq:subdiff}. Reordering yields
\begin{align*}
\langle \nabla f(\bar{u}), \bar{u} - u \rangle + g(\bar{u}) - g(u) \leq 0  \quad \forall u \in \mathcal{U}.
\end{align*}
Maximizing with respect to~$u\in \mathcal{U}$, on both sides we conclude~$\Psi(\bar{u})\leq 0$. Since~$\Psi$  assumes non-negative values only, the statement follows.
Conversely if~$\bar{u}$ fulfills~$\Psi(\bar{u})=0$ we readily obtain
\begin{align*}
\langle \nabla f(\bar{u}), \bar{u}\rangle +g(\bar{u}) \leq \langle \nabla f(\bar{u}), u\rangle +g(u) \quad  \forall u \in \mathcal{U}
\end{align*}
By rearranging both sides we arrive at~\eqref{eq:subdiff}. Next we prove~\eqref{eq:boundforres}. Let~$u \in \dom j$ be given and let~$\optu$ denote a minimizer of~\eqref{def:problem}. From the convexity of~$f$ we readily obtain
\begin{align*}
j(u)-\min_{u \in \mathcal{U}}j(u)=j(u)-j(\bar{u}) \leq \langle \nabla f(u),u-\bar{u}\rangle+ g(u)-g(\bar{u}).
\end{align*}
The right hand side is estimated by
\begin{align*}
\langle \nabla f(u),u-\bar{u}\rangle+ g(u)-g(\bar{u})\leq  \sup_{v\in \mathcal{U}} \left \lbrack \langle \nabla f(u),u-v\rangle+ g(u)-g(v) \right\rbrack=\Psi(u),
\end{align*}
and thus \eqref{eq:boundforres} follows.

Finally, let a sequence~$u^k \rightharpoonup^* u$ be given.
For each $v\in \mathcal{U}$ we obtain
\begin{align*}
\Psi(u^k)\geq \langle \nabla f(u^k),u^k-v\rangle+g(u^k)-g(v).
\end{align*}
Taking the limes inferior for $k \rightarrow \infty$ on both sides of the inequality yields
\begin{align*}
\liminf_{k\rightarrow \infty}\Psi(u^k)\geq \langle \nabla f(\bar{u}),\bar{u}-v\rangle+g(\bar{u})-g(v),
\end{align*}
due to the weak* convergence of $u^k$ and the continuity properties of $\nabla f$ and $g$ in $\mathbf{A1}$ and~$\mathbf{A2}$. Since $v$ was chosen arbitrarily, we can maximize over all $v \in \mathcal{U}$ from which we conclude $\liminf_{k \rightarrow \infty} \Psi(u_k)\geq \Pi(\bar{u})$.
\end{proof}
\section{Generalized Frank-Wolfe methods with backtracking stepsize} \label{sec:generalizedfrankwolfe}
This section is devoted to the numerical solution of Problem~\eqref{def:problem}. For this purpose we adapt the generalized conditional gradient (GCG) method from~\cite{bredies} to the problem at hand. The proposed algorithm iteratively generates a sequence of iterates~$u^k$,~$k\in\N$, by alternating between solving a (partially) linearized problem to obtain a descent direction~$v^k$ and forming a convex combination to get the subsequent iterate:
\begin{align} \label{def:gcg}
v^k \in \argmin_{v \in \mathcal{U}} \left \lbrack \langle \nabla f(u^k),v \rangle+ g(v) \right \rbrack,~u^{k+1}=(1-s^k) u^k+ s^k v^k.
\end{align}
Here,~$s^k \in (0,1)$ is a suitably chosen stepsize. This iteration is repeated as long as~$\Psi(u^k)>0$, see Theorem~\ref{thm:optimality}. Note that if~$\mathcal{U}_{ad} \subset \mathcal{U}$ is a convex, compact (in a suitable sense) set and~$g(u)=I_{\mathcal{U}_{ad}}(u)$ is the associated indicator function, then~\eqref{def:gcg} reduces to the "classical" conditional gradient method described in~\cite{frank}. In the following, we briefly go over the individual steps of the method in more detail. First, the choice of the descent direction~$v^k$ according to~\eqref{def:gcg} is always well-defined.
\begin{proposition} \label{prop:extofvk}
Let~$u \in \mathcal{U}$ be arbitrary. Then there exists at least one minimizer to
\begin{align}
\label{def:linprob}
\min_{v \in \mathcal{U}} \left \lbrack \langle \nabla f(u),v \rangle+ g(v) \right \rbrack \tag{$\mathcal{P}_{\text{lin}}(u)$}
\end{align}
Moreover, for every~$u^0 \in \dom j$, the set
\begin{align} \label{def:setoflinsol}
S(u)=\left\{\,v \in \argmin \eqref{def:linprob}\;|\;u \in E_j(u^0)\,\right\}
\end{align}
is bounded.
\end{proposition}
\begin{proof}
For abbreviation set~$h(v)=\langle \nabla f(u),v \rangle+ g(v)$. Note that~$h$ is proper. In particular this implies~$\inf_{v \in \mathcal{U}} h(v)<+\infty$. Let~$v_k \in \mathcal{U}$,~$k\in\N$, denote a minimizing sequence for~$h$ i.e
\begin{align*}
h(v_k) \rightarrow \inf_{v \in \mathcal{U}} h(v).
\end{align*}
Now assume that~$\unorm{v_k}$ is unbounded. Then we have~$\unorm{v_k}>0$ for all~$k$ large enough as well as
\begin{align*}
\inf_{v \in \mathcal{U}} h(v)= \lim_{k\rightarrow \infty} h(v_k) \geq \limsup_{k\rightarrow \infty} \left\lbrack (g(v_k)/\unorm{v_k}-\cnorm{\nabla f(u)})\unorm{v_k} \right \rbrack =+\infty
\end{align*}
due to the coercivity of~$g$. This contradicts the fact that~$h$ is proper. Hence~$\unorm{v_k}$ is bounded and, due to the sequential Banach-Alaoglu theorem, there exists a subsequence, denoted by the same subscript, as well as~$\bar{v} \in \mathcal{U}$ with~$v_k \rightharpoonup^* \bar{v}$. Finally due to weak* lower semicontinuity of~$g$ we conclude
\begin{align*}
\inf_{v \in \mathcal{U}} h(v)=\liminf_{k \rightarrow \infty} h(v_k) \geq h(\bar{v})
\end{align*}
which shows that~$\bar{v}$ is a minimizer of~\eqref{def:linprob}.

Now let~$u^0 \in \text{dom}(j)$ be arbitrary but fixed. Recall that~$E_j(u^0)$ is weak* compact, see Proposition \ref{prop:existence}, and that ~$\nabla f$ is weak*-to-strong continuous. Hence there is~$M>0$ with~$\cnorm{\nabla f(u) } \leq M$ for all~$u \in E_j(u^0)$. Moreover note that
\begin{align*}
\min_{v \in \mathcal{U}} \left \lbrack \langle \nabla f(u),v \rangle+ g(v) \right \rbrack \leq g(0) <+\infty
\end{align*}
i.e. the minimum of~\eqref{def:linprob} is bounded independently of~$u \in E_j(u^0) $. Now assume that there are sequences~$v_k \in S(u)$,~$u_k \in \mathcal{U}$,~$k\in\N$, with~$\unorm{v_k} \rightarrow +\infty$ and
\begin{align*}
\langle \nabla f(u_k),v_k \rangle+ g(v_k)=\min_{v \in \mathcal{U}}\left \lbrack \langle \nabla f(u_k),v \rangle+ g(v) \right\rbrack.
\end{align*}
Then, similarly as before, we observe that
\begin{align*}
+\infty > g(0)\geq \limsup_{k \rightarrow \infty} \left\lbrack \langle \nabla f(u_k),v_k \rangle+ g(v_k) \right \rbrack \geq \limsup_{k \rightarrow \infty}
\left\lbrack (g(v_k)/\unorm{v_k}-M)\unorm{v_k} \right\rbrack=+\infty
\end{align*}
which gives a contradiction. Since the sequence~$v_k$ was chosen arbitrarily, $S(u)$ is bounded. This finishes the proof.
\end{proof}
In many challenging applications the descent direction~$v^k$ can be computed at a neglectable computational effort or can be even given analytically. This is discussed for two examples in the context of PDE constrained optimization in detail in Section~\ref{sec:applications}.

Second, we shed light on the choice of the stepsize~$s^k \in (0,1)$ in~\eqref{def:gcg}.
Of course, the properties of the resulting algorithm as well as its computational complexity depend crucially on its particular selection. The focus of this manuscript lies on a nonsmooth version of the well-known Armijo condition which is solely based on repeated function evaluations of~$j$.

\begin{mydef} \label{def:Quasiarmijogoldstein}
Let~$\gamma \in (0,1)$, $\alpha \in (0,1/2]$ be fixed.
The step size $s^k$ is chosen according to the \textit{Quasi-Armijo-Goldstein condition}, i.e. ~$s^k=\gamma^{n_k}$ where~$n_k \in \N$ is the smallest integer with
 \begin{align}\label{Def:Armijo}
\alpha \gamma^{n_k} \Psi(u^k)\leq j(u^k)-j(u^k+\gamma^{n_k}(v^k-u^k)).
\end{align}
\end{mydef}
The following lemma illustrates that this choice of the step size is always possible if~$u^k$ is not a stationary point. Moreover it establishes that the Quasi-Armijo-Goldstein stepsize guarantees descent of the objective functional value in every iteration, unless $\Psi(u_k)=0$.
\begin{lemma}\label{lem:possibilityofarmijo}
Let an arbitrary measure $u \in \dom j$ be given. Assume that $\Psi(u)>0$ and denote by $\bar{v}$ the solution of the associated partially linearized problem \eqref{def:linprob}. Define $u_s=u+s(\bar{v}-u)$ as well as
\begin{align*}
W:~(0,1] \rightarrow \mathbb{R} \quad W(s)=\frac{j(u)-j(u_s)}{s \Psi(u)}.
\end{align*}

The function $W$ is upper left semicontinuous,~$W \in \mathcal{C}((0,1);\mathbb{R})$, and
$\liminf_{s\rightarrow 0} W(s)=1$. In particular, if~$u^k$ is generated by Algorithm~\ref{alg:gcg} and~$\Psi(u^k)>0$, then there is a smallest~$n_k \in \mathbb{N}$ such that \eqref{Def:Armijo} is satisfied and, consequently,
\begin{align*}
j(u^{k+1})-j(u^k) < 0.
\end{align*}
for~$u^{k+1}=u^k+\gamma^{n_k}(v^k-u^k)$.
\end{lemma}
\begin{proof}
Due to the definition of $\bar{v}$ we have
\begin{align*}
W(s)=\frac{j(u)-j(u_s)}{s \Psi(u)}=\frac{j(u)-j(u_s)}{s \left( \pair{\nabla f(u), u - \bar{v}} + g(u) - g(\bar{v})\right)}.
\end{align*}
for~$s>0$.
Note that the mapping~$s \mapsto j(u_s)$ is convex and  lower semicontinuous on its domain~$[0,1]$. Hence it is continuous on~$(0,1)$,~\cite[Proposition 2.5]{ekeland}, and we have~$W \in \mathcal{C}((0,1);\mathbb{R})$. Moreover we also get the upper left semicontinuity of~$W$.
From the mean value theorem we get for every $s\in (0,1)$ the existence of $\zeta_s \in [0,1]$ and $\tilde{u}_s= u+\zeta_s(u_s-u)\in \dom j$ with
\begin{align*}
W(s)=\frac{ s\langle\nabla f(\tilde{u}_s), u - \bar{v} \rangle + g(u) - g(u_s)}{s \left( \langle \nabla f(u), u - \bar{v}\rangle + g(u) - g(\bar{v})\right)}
\end{align*}
Using the convexity of $g$, we estimate
\begin{align*}
W(s) =\frac{ s\langle\nabla f(\tilde{u}_s), u - \bar{v}\rangle + g(u) - g(u_s)}{s \left( \langle\nabla f(u), u - \bar{v}\rangle + g(u) - g(\bar{v})\right)} \geq
\frac{ s \left( \langle\nabla f(\tilde{u}_s), u - \bar{v}\rangle + g(u) - g(\bar{v})\right)}{s \left( \langle\nabla f(u), u - \bar{v}\rangle + g(u) - g(\bar{v})\right)}.
\end{align*}
Since $\zeta_s$ is bounded independently of $s$, there holds $\tilde{u}_s \rightharpoonup^* u$ for $s\rightarrow 0$. Due to the weak*-to-strong continuity of $\nabla f$, the right-hand side of the inequality tends to $1$ as $s \to 0$, yielding $\liminf_{s\rightarrow 0} W(s) \geq 1$. The statement on the existence of~$n_k$ follows immediately due to ~$W \in \mathcal{C}((0,1); \mathbb{R})$ and~$\liminf_{s\rightarrow 0} W(s)\geq 1$. Finally we get
\begin{align*}
j(u^{k+1})-j(u^k) \leq - \gamma^{n_k} \Psi(u^k) <0
\end{align*}
due to Definition~\ref{def:Quasiarmijogoldstein} and~$\Psi(u^k),~\gamma^{n_k}>0$.
\end{proof}
\begin{remark}
Let us point out that sublinear rates of convergence for the GCG method have already been proven for a large variety of other stepsize choices. We refer e.g. to "optimal" steps,~\cite{bredies,xu2017}, i.e.
\begin{align*}
s^k \in \argmin_{s \in[0,1]} j(u^k+s(v^k-u^k)),
\end{align*}
or adaptive "short steps",~\cite{beck}. Note that, in comparison to Definition~\ref{def:Quasiarmijogoldstein}, computing these stepsizes either requires the solution of a one-domensional minimization problem or explicit knowledge of additional paramaters e.g.~the Lipschitz constant~$L_{u^0}$ of~$\nabla f$. For the "classical" CG method,
i.e. for~$g=I_{\mathcal{U}_{ad}}$,
improved rates of convergence for those types of stepsizes are e.g. discussed in~\cite{Dunn79,kerdreux,guelat,levitin}. For the GCG method we are not aware of any results in this direction. However, they can be easily derived using, mutatis mutandis, the same techniques applied in the following sections. We thus omit a detailed discussion in this paper. Finally we also mention open loop stepsizes, e.g.~$s^k=2/(k+1)$, see e.g.~\cite{yu}. While these are simple to implement, a sublinear rate of convergence is essentially the best one obtainable for a GCG method with this kind of stepsize,~\cite{dunnopen,Dunn79}.
\end{remark}
Finally we note that both, choosing~$s^k$ according to Definition~\ref{Def:Armijo} as well as checking for convergence of the method, requires the computation of the gap functional~$\Psi(u^k)$. The following lemma illustrates that this is easily obtained as a byproduct of solving~\eqref{def:linprob}.
\begin{lemma} \label{lem:easydualgap}
Let~$u \in \dom g$ be arbitrary but fixed and denote by~$\bar{v}$ a solution to~\eqref{def:linprob}. Then there holds
\begin{align*}
\Psi(u)= \left \lbrack \langle \nabla f(u),u-\bar{v} \rangle+g(u)-g(\bar{v}) \right\rbrack
\end{align*}
\end{lemma}
\begin{proof}
This result follows immediately from the definitions of~$\Psi(u)$, see~\eqref{def:gap}, and~\eqref{def:linprob}.
\end{proof}
For further reference the method is again summarized in Algorithm~\ref{alg:gcg}.
\begin{algorithm}
\begin{flushleft}
\hspace*{\algorithmicindent} \textbf{Input:} Starting point~$u^0 \in \mathcal{U}$,~$\gamma \in (0,1)$,~$\alpha\in (0, 1/2\rbrack$. \\
\hspace*{\algorithmicindent} \textbf{Output:} Minimizer~$\optu$ to~\eqref{def:problem}.
\end{flushleft}
\begin{algorithmic}
\STATE
\WHILE{$\Psi(u^k)>0$}
\STATE
\STATE 1. Compute a solution~$v^k \in \mathcal{U}_{ad}$ to
\begin{align*}
\min_{v \in \mathcal{U}} \left \lbrack \langle \nabla f(u^k), v \rangle+g(v) \right \rbrack.
\end{align*}
\STATE

\STATE 2. Choose~$s^k \in [0,1]$ according to Definition~\ref{def:Quasiarmijogoldstein}, update~$u^{k+1}=u^k+ s^k(v^k-u^k)$.
\STATE

\STATE 3. Set~$k=k+1$ and go to the next iteration.

\ENDWHILE
\end{algorithmic}
\caption{Generalized Frank-Wolfe method for~\eqref{def:problem}} \label{alg:gcg}
\end{algorithm}

\section{Convergence Analysis} \label{sec:convergence}
This section addresses the convergence properties of Algorithm~\ref{alg:gcg}. In particular, our interest lies in deriving convergence rates for the \textit{residuals}
\begin{align*}
r_j(u^k) \coloneqq j(u^k)- \min_{u \in \mathcal{U}} j(u)
\end{align*}
associated to the sequence of iterates~$u^k$. For this purpose we proceed in two steps. First, in Section~\ref{subsec:sublinconv}, we show that the sequence~$u^k$ converges (on subsequences) towards minimizers of~\eqref{def:problem}. Moreover there holds~$r_j(u^k)= \mathcal{O}(1/(k+1))$.
While it is known that this sublinear rate of convergence is tight for general problems,~\cite{cannon}, we will improve on it, in a second step, if there is a minimizer~$\optu$ which satisfies
\begin{align} \label{eq:strenghtfirstorder}
 \langle  \nabla f(\optu),u- \optu \rangle+g(u)-g(\optu) \geq \theta \, \norm{u-\optu}^q_* \quad \forall u \in \mathcal{U}
\end{align}
for some~$\theta>0$ and an exponent~$q \geq 2$. Indeed, if~$q=2$, Algorithm~\ref{alg:gcg} converges linearly. Otherwise, for~$q>2$, improved sublinear rates are obtained.
Recalling the variational inequality~\eqref{eq:subdiff},~\eqref{eq:strenghtfirstorder} can be interpreted as a~\textit{strengthened first order condition}. In the case~$q=2$, it is intimidely related to the~\textit{strong metric subregularity} of the subdifferential~$\partial g$ at~$\optu$ with respect to~$-\nabla f(\optu)$,~\cite{metricsub}. Due to the convexity of~$f$,~\eqref{eq:strenghtfirstorder} implies a growth condition of the form
\begin{align*}
r_j(u)=j(u)-j(\optu) \geq \theta \, \stnorm{u-\optu}^q \quad \forall u \in \mathcal{U}
\end{align*}
for~$j$. This also allows to derive convergence rates for the sequence of iterates $u_k$. We stress that both, the improved rates for~$r_j(u^k)$ as well as the results for~$u^k$, solely rely on~\eqref{eq:strenghtfirstorder} but~\textit{donot} require any uniform convexity assumptions of~$f$.

Of course,~\eqref{eq:strenghtfirstorder} represents an additional, non-trivial, requirement which has to be discussed for each instance of Problem~\eqref{def:problem} individually. Several examples from the literature on PDE-constrained optimal control are discussed in Section~\ref{sec:applications}. Moreover, for the "classical" CG method, inequalities of the form~\eqref{eq:strenghtfirstorder} are implied by the local uniform convexity of the constraint set~$\mathcal{U}_{ad}$ at~$\optu$. For a reference we refer to~\cite{Dunn79,kerdreux}.

\subsection{Auxiliary Results}
\label{subsec:auxiliary}
We start by collecting some useful auxiliary results. In the analysis of first order methods, a vital role is played by the classical descent lemma. For a reference we point  to~\cite{bertsekas}. In the following, we require a nonsmooth version of this standard result.
\begin{lemma} \label{lem:descentlem}
Let Assumptions~\ref{ass:functions} and~\ref{ass:lipschitz} hold.
Assume that~$u^k$ and~$v^k$ are generated by Algorithm~\ref{alg:gcg}. For~$s \in [0,1]$ set~$u^k_s= u^k+s (v^k-u^k)$. If~$u^k_s \in E_j(u^0)$, then we have
\begin{align*}
r_j(u^k_s)-r_j(u^k) \leq  -s\Psi(u^k)+\frac{L_{u_0}}{2} \left(s\stnorm{u^k-v^k}\right)^2.
\end{align*}
\end{lemma}
\begin{proof}
Due to the convexity of the sublevel set~$E_j(u^0)$ we obtain
\begin{align*}
j(u^{k}_s)-j(u^{k})= &-s \langle \nabla f(u^k), u^k -v^k \rangle + g(u^{k}_s)-g(u^k)+ \int_0 ^{s} \! \langle \nabla f(u_{\sigma})-\nabla f(u^k),v^k-u^k\rangle~ \mathrm{d} \sigma ,
\end{align*}
with $u_{\sigma}=u^k+\sigma (v^k-u^k)\in E_{j}(u^0)$ for $\sigma\in[0,s]$.
Using the convexity of $g$ we obtain
\begin{align*}
-s \langle \nabla f(u^k), u^k -v^k \rangle + g(u^{k}_s)-g(u^k)
 \leq
-s \left(\langle \nabla f(u^k), u^k -v^k \rangle +g(u^k)-g(v^k)\right),
\end{align*}
where the right-hand side simplifies to $-s \Psi(u^k)$ due to the definition of~$v^k$, see Lemma~\ref{lem:easydualgap}.
Finally, invoking Assumption~\ref{ass:lipschitz}, we get
\begin{align*}
\int_0 ^{s} \! \langle \nabla f(u_{\sigma})-\nabla f(u^k),v^k-u^k\rangle ~\mathrm{d} \sigma &   \leq   \frac{L_{u^0}}{2} \left(s \stnorm{v^k -u^k}\right)^2.
\end{align*}
Combining both estimates and noting that~$r_j(u^k) \leq \Psi(u^k)$ yields the proof.
\end{proof}
We particularly point out that the per-iteration descent of Algorithm~\ref{alg:gcg} is governed by two terms playing off against each other: The negative quantity~$- \Psi(u^k)$ and the positive distance between~$v^k$ and~$u^k$. Hence, intuitively, the convergence rate of Algorithm~\ref{alg:gcg} is dominated by the asymptotic behaviour of both terms in relation to each other. This is formalized in the following proposition.
\begin{proposition} \label{prop:recursion}
Let~$u^k$ be generated by Algorithm~\ref{alg:gcg} and assume that~$\Psi(u^k)>0$ for all~$k\in\N$. Moreover let~$q_k$,~$k\in \N$, denote a sequence with
\begin{align}\label{defqk}
\alpha \, \min \left\{\frac{2(1-\alpha)\gamma\Psi(u^k)}{L_{u^0}\, \stnorm{v_k-u_k}^2}, 1\right\} \geq q_k >0.
\end{align}
Then there holds
\begin{align} \label{eq:descentqk}
r_j(u^{k+1}) \leq \left(1- q_k\right)\,  r_j(u^{k})
\end{align}
for all~$k \in \N$.
%\begin{align} \label{defqk}
%q_k \coloneqq r_j(u^0)\,\alpha\gamma \, \min \left\{\frac{2(1-\alpha)}{L_{u^0}\, \stnorm{v_k-u_k}^2}, \frac{1}{r_j(u^k)}\right\}.
%\end{align}
\end{proposition}
\begin{proof}
By definition of the step size $s^k$, see Definition~\ref{def:Quasiarmijogoldstein}, as well as~\eqref{eq:boundforres} there holds
\begin{align*}
\alpha s^k r_j(u^k)\leq \alpha s^k \Psi(u^k)\leq r_j(u^k)-r_j(u^{k+1}),
\end{align*}
which yields
\begin{align} \label{descentinresidual}
r_j(u^{k+1})\leq(1-\alpha s^k)\, r_j(u^k).
\end{align}
Thus, to arrive at~\eqref{defqk} we only require a suitable estimate on~$\alpha s^k$.
Since $\Psi(u^k)>0$ we obtain $0 < s^k \leq  1$ for all $k$. Two cases have to be distinguished. If $s^k=1$ we immediately arrive at
\begin{align} \label{eq:recurseest1}
r_j(u^{k+1})\leq(1-\alpha ) \,r_j(u^k).
\end{align}
In the second case, if $s^k< 1$, there exists $\hat{s}^k \in [s^k, s^k / \gamma)$ with
\begin{align*}
\alpha=\frac{j(u^k)-j(u^k +\hat{s}^k(v^k-u^k))}{\hat{s}^k \Psi(u^k)},
\end{align*}
using that the function W from Lemma~\ref{lem:possibilityofarmijo} is in~$\mathcal{C}(0,1)$ and applying the intermediate value theorem,~\cite{guillerme}. Consequently, $u^k +s(v^k-u^k)\in E_j (u^0)$ for all $0\leq s\leq \hat{s}^k$ due to the convexity of $j$.
Now, because of the Lipschitz-continuity of $\nabla f$ on $E_j (u^0)$, Lemma \ref{lem:descentlem} can be applied and, defining $\delta u^k =v^k-u^k$, there holds
\begin{align*}
\alpha = \frac{j(u^k)-j(u^{k}+\hat{s}^k\delta u^k)}{\hat{s}^k\Psi(u^k)}
\geq 1- \frac{L_{u^0} \hat{s}^k}{2}  \frac{\stnorm{\delta u^k}^2}{\Psi(u^k)}
\geq 1- \frac{L_{u^0} s^k}{2\gamma}  \frac{\stnorm{\delta u^k}^2}{\Psi(u^k)}.
\end{align*}
In the last inequality we have used that $\hat{s}^k \leq s^k/\gamma$. Moreover we point out that~$\delta u^k \neq 0$ since~$\Psi(u^k)>0$.
Reordering yields
\begin{align*}
1\geq s^k \geq 2 \gamma (1-\alpha) \frac{\Psi(u^k)}{L_{u^0} \stnorm{v^k -u^k}^2}.
\end{align*}
and thus by \eqref{descentinresidual}
\begin{align} \label{eq:recurseest2}
r_j(u^{k+1})\leq \left(1-\alpha \gamma \,\frac{ 2  (1-\alpha)\Psi(u^k)}{L_{u^0} \stnorm{v^k -u^k}^2} \right) \,r_j(u^k).
\end{align}
Combining both estimates,~\eqref{eq:recurseest1} and~\eqref{eq:recurseest2}, the inequality
%\begin{align}\label{recursion}
%0\leq \frac{r_j(u^{k+1})}{r_j(u^0)}\leq \frac{r_j(u^{k+1/2})}{r_j(u^0)}\leq \frac{r_j(u^{k})}{r_j(u^0)}-q_k \left ( \frac{r_j(u^{k})}{r_j(u^0)}\right )^2\quad \forall k\in \mathbb{N}
%\end{align}
\begin{align}\label{recursion}
0\leq r_j(u^{k+1})\leq r_j(u^{k+1/2})\leq (1-q_k)\, r_j(u^k)\quad \forall k\in \mathbb{N}
\end{align}
holds where $q_k$ satisfies~\eqref{defqk}.
%\begin{align*}
%q_k=r_j(u^0)\, \alpha \gamma \, \min \left \{ \frac{2  (1-\alpha)}{L_{u^0} \stnorm{v^k -u^k}^2} ,\frac{1}{r_j(u^k)}\right \}.
%\end{align*}
\end{proof}

\subsection{Sublinear convergence rates} \label{subsec:sublinconv}
Invoking the last proposition, it becomes clear that worst case convergence rates for Algorithm~\ref{alg:gcg} can be derived by choosing a suitable sequence~$q_k>0$ according to~\eqref{defqk} and studying its asymptotic behaviour. For example, without further assumptions on problem~\eqref{def:problem}, the following choice is possible.
\begin{lemma} \label{lem:boundofvkanduk}
Let~$u^k$ and~$v^k$,~$k\in\N$, be generated by Algorithm~\ref{alg:gcg}. Then there is~$M_*>0$ with~$\max \{\norm{v^k}_*, \norm{u^k}_*\}\leq M_*$ for some~$M_*>0$. In particular~$q_k$ in~\eqref{defqk} can be chosen as
\begin{align}
\label{defqkslow}
q_k \coloneqq \alpha\, \min \left\{\frac{(1-\alpha)\gamma r_j(u^k)}{2L_{u^0}\,M^2_* }, 1\right\}  >0.
\end{align}
\end{lemma}
\begin{proof}
From Lemma~\ref{lem:possibilityofarmijo} we readily get~$u^k \in E_j(u^0)$,~$k\in\N$. Now, due to Proposition~\ref{prop:existence} and~\ref{prop:extofvk}, $u^k$ and~$v^k$ are bounded w.r.t.~$\unorm{\cdot}$ independently of~$k\in\N$. Together with Assumption~\ref{ass:lipschitz}~$\mathbf{A4}$ this gives the existence of~$M_*>0$ with~$\max \{\norm{v^k}_*, \norm{u^k}_*\}\leq M_*$ as claimed. Finally we invoke~\eqref{eq:boundforres} to get
\begin{align*}
\alpha \, \min \left\{\frac{2(1-\alpha)\gamma \Psi(u^k)}{L_{u^0}\, \stnorm{v_k-u_k}^2}, 1\right\} \geq \alpha \, \min \left\{\frac{(1-\alpha)\gamma r_j(u^k)}{2L_{u^0}\, M^2_*}, 1\right\}.
\end{align*}
\end{proof}

Together with the following recursion lemma this choice of~$q_k$ guarantees a sublinear~$\mathcal{O}(1/(k+1))$ rate of convergence for Algorithm~\ref{alg:gcg}.
\begin{lemma}
\label{lem:recursion1}
Let~$h_k > 0$,~$k \in \N$, denote a sequence of numbers which satisfy~$h_0=1$ as well as
\begin{align*}
h_{k+1} \leq  h_k- q h^2_k
\end{align*}
for all~$k \in \N$ and some~$q>0$.
Then we also have
\begin{align*}
0 <h_k \leq \frac{1}{1+ qk}
\end{align*}
for all~$k \in \N$.
\end{lemma}
\begin{proof}
See~\cite[Lemma 5.1]{Dunn79}.
\end{proof}
\begin{theorem}
\label{thm:slowconvergence}
Let Assumptions~\ref{ass:functions} and~\ref{ass:lipschitz} hold. Then Algorithm~\ref{alg:gcg} either terminates after~$k \in \N$ iterations with~$u^k$ a minimizer to~\eqref{def:problem} or it generates a sequence~$u^k$,~$k\in \N$. In the second case we have
\begin{align} \label{eq:slowrate}
r_j(u^k) \leq \frac{r_j(u^0)}{1+qk}, \quad q= \alpha  \min\left\{\frac{(1-\alpha)\gamma r_j(u^0)}{2L_{u^0}M^2_*},1\right\}.
\end{align}
Moreover there exists a weak* accumulation point~$\optu$ of~$u^k$,~$k\in\N$, and
every such point is a global minimum of j. Finally, if the solution~$\optu$ to~\eqref{def:problem} is unique, then we have~$u^k \rightharpoonup^* \optu$ for the whole sequence.
\end{theorem}
\begin{proof}
If Algorithm~\ref{alg:gcg} terminates after~$k$ iterations then we have~$\Psi(u^k)=0$. Consequently~$u^k$ is a minimizer to~\eqref{def:problem}, see Theorem~\ref{thm:optimality}, and nothing more is to prove. Now assume that Algorithm~\ref{alg:gcg} does not terminate after finitely many iterations and recall the definition of~$q_k>0$ from~\eqref{defqkslow}. Then we estimate
\begin{align*}
q_k r_j(u^k)  &\geq \alpha  \min\left\{\frac{(1-\alpha)\gamma r_j(u^0)}{2L_{u^0}M^2_*},\frac{r_j(u^0)}{r_j(u^k)}\right\} \frac{r_j(u^k)^2}{r_j(u^0)} \\
&\geq \alpha \min\left\{\frac{(1-\alpha) \gamma r_j(u^0)}{2L_{u^0}M^2_*},1\right\} \frac{r_j(u^k)^2}{r_j(u^0)}
\end{align*}
Hence, divding by~$r_j(u^0)$ in~\eqref{eq:descentqk}, yields
\begin{align*}
\frac{r_j(u^{k+1})}{r_j(u^0)} \leq \frac{r_j(u^k)}{r_j(u^0)}- q \left(\frac{r_j(u^k)^2}{r_j(u^0)} \right)^2.
\end{align*}
with~$q$ defined as in~\eqref{eq:slowrate}.
The claimed convergence rate is then obtained directly from Lemma~\ref{lem:recursion1}.

Now, since~$u^k$ is bounded with respect to~$\unorm{\cdot}$, it admits at least one weak* convergent subsequence denoted by the same subscript with limit~$\optu$. Due to the weak* lower semicontinuity of~$j$ and the definition of the residual~$r_j(u^k)$ we get
\begin{align*}
0 \leq r_j(\optu) \leq \liminf_{k \rightarrow \infty} r_j(u^k)=0.
\end{align*}
Hence~$r_j(\optu)=0$ and~$\optu$ is a minimizer to~\eqref{def:problem}. Finally if the minimizer~$\optu$ to~\eqref{def:problem} is unique, then every~weak* convergent subsequence of~$u^k$ converges to~$\optu$ which yields weak* convergence of the whole sequence. This finishes the proof.
\end{proof}
\subsection{Improved convergence rates under strengthend first order conditions} \label{subsec:fastconv}
Next we will improve on the convergence result of Theorem~\ref{thm:slowconvergence} provided that the following additional assumptions on Problem~\eqref{def:problem} hold:
\begin{assumption} \label{ass:fastconv}
\phantom{ d}
\begin{itemize}
\item[\textbf{B1}] Assumptions~\ref{ass:functions} and~\ref{ass:lipschitz} hold.
\item[\textbf{B2}] There is a minimizer~$\optu$ to which satisfies
\begin{align*}
\langle  \nabla f(\optu),u- \optu \rangle+g(u)-g(\optu) \geq \theta \, \norm{u-\optu}^q_* \quad \forall u \in \mathcal{U}
\end{align*}
 for some~$\theta>0$ and~$q\geq 2$.
\end{itemize}
\end{assumption}
We start by noting that~$\optu$ from Assumption~\ref{ass:fastconv}~\textbf{B2} is indeed the unique solution to~\eqref{def:problem} and~$(\mathcal{P}_{\text{lin}}(\optu))$.
\begin{lemma} \label{lem:estonukvk}
Let Assumption~\ref{ass:fastconv}~\textbf{B2} hold for some~$\optu$. Then~$\optu$ is the unique solution to~\eqref{def:problem} and~$(\mathcal{P}_{\text{lin}}(\optu))$.
\end{lemma}
\begin{proof}
Let us assume that there is a second minimizer~$\optu_2 \neq \optu$ to~\eqref{def:problem}. Due to the convexity of~$f$ we get
\begin{align*}
0=j(\optu_2)-j(\optu) \geq \langle  \nabla f(\optu),\optu_2- \optu \rangle+g(\optu_2)-g(\optu) \geq \theta \, \norm{\optu_2-\optu}^q_*>0
\end{align*}
which yields a contradiction. Hence the minimizer of~\eqref{def:problem} is unique. Further $\Psi(\bar u)=0$ and hence $\bar u$ is a minimizer for $(\mathcal{P}_{\text{lin}}(\optu))$. As a consequence of \textbf{B2} it is also unique.
\end{proof}
Now, before proceeding to the main results of this section, let us give some intuition to the role of Assumption~\ref{ass:fastconv} in the following investigations. For this purpose recall the descent inequality
\begin{align*}
r_j(u^{k+1})-r_j(u^k) \leq  -s^k\Psi(u^k)+\frac{L_{u_0}}{2} \left(s^k\stnorm{u^k-v^k}\right)^2.
\end{align*}
from Lemma~\ref{lem:descentlem}. As already discussed in Proposition~\ref{prop:recursion}, the per-iteration descent of Algorithm~\ref{alg:gcg} depends on the asymptotic behaviour of~$\Psi(u^k)$ in relation to~$\stnorm{v^k-u^k}$. In general, the "bad" term~$\stnorm{v^k-u^k}$ does not converge to zero. This is due to the (potential) nonuniqueness of the minimizer to the linearized problem~\eqref{def:linprob}.
Using Assumption~\ref{ass:fastconv}, however, we cannot only show that~$v^k \rightarrow \optu$,~$u^k \rightarrow \optu$ but also get a quantitative estimate for~$\stnorm{v^k-u^k}$ in terms of the residual.
\begin{lemma} \label{lem:estsforcontrol}
Let Assumption~\ref{ass:fastconv}~$u^k$ hold  and let ~$v^k$,~$k \in \N$, be generated by Algorithm~\ref{alg:gcg}. Then there holds
\begin{align} \label{eq:contest1}
\stnorm{v^k-\optu} \leq \left(\frac{L_{u^0}}{\theta} \right)^{\frac{1}{(q-1)}} \,\stnorm{u^k-\optu}^{\frac{1}{(q-1)}},~\stnorm{u^k-\optu} \leq \left(\frac{1}{\theta}\right)^{\frac{1}{q}}\, r_j(u^k)^{\frac{1}{q}}
\end{align}
In particular, this implies
\begin{align} \label{eq:contest2}
\stnorm{u^k-v^k} \leq c_1 \, r_j(u^k)^{\frac{1}{q}} + c_2 \,r_j(u^k)^{\frac{1}{q(q-1)}}.
\end{align}
for some~$c_1,c_2 >0$ independent of~$k\in \N$ as well as
\begin{align} \label{eq:convofukvk}
\lim_{k \rightarrow \infty} \left \lbrack \stnorm{v^k-u^k}+\stnorm{v^k-\optu}+\stnorm{u^k-\optu}\right \rbrack =0.
\end{align}
\end{lemma}
\begin{proof}
Recall the definition of the dual gap~$\Psi(u^k)\geq 0$, see~\eqref{def:gap}.
Using Assumption~\ref{ass:fastconv}~$\mathbf{B2}$ , the choice of $v_k$ as minimizer to {$(\mathcal{P}_{\text{lin}}(u_k))$}, and  $\mathbf{A5}$ we estimate
 \begin{align*}
\theta \,\stnorm{v^k-\bar{u}}^{q} &\leq \langle \nabla f(\bar{u}),v^k-\optu \rangle + g(v^k)-g(\optu) \\ & \leq
\langle \nabla f(\bar{u})-\nabla f(u^k),v^k-\optu \rangle + \langle \nabla f(u^k),v^k-u^k \rangle + g(v^k)-g(\optu)
\\ & \leq \langle \nabla f(\bar{u})-\nabla f(u^k),v^k-\optu \rangle \leq L_{u^k}\, \stnorm{u^k-\bar{u}}\,\stnorm{(v^k-\bar{u}}.
\end{align*}
Noting that~$L_{u^0} \geq L_{u^k} $ proves the first inequality in~\eqref{eq:contest1}. Now, by convexity of~$f$, we get
\begin{align*}
r_j(u^k) \geq \langle \nabla f(\optu), u^k-\optu \rangle+g(u^k)-g(\optu) \geq \theta \,\stnorm{u^k-\optu}^q.
\end{align*}
Dividing both sides by~$\theta$ and taking the root finishes the proof of~\eqref{eq:contest1}. The inequality in~\eqref{eq:contest2} is readily obtained by the triangle inequality:
\begin{align*}
\stnorm{u^k-v^k} \leq \stnorm{u^k-\optu}+\stnorm{v^k-\optu} &\leq \stnorm{u^k-\optu}+\left(\frac{L_{u^0}}{\theta} \right)^{\frac{1}{(q-1)}} \,\stnorm{u^k-\optu}^{\frac{1}{(q-1)}} \\
&\leq \stnorm{u^k-\optu}+\left(\frac{L_{u^0}}{\theta} \right)^{\frac{1}{(q-1)}} \left(\frac{1}{\theta}\right)^{\frac{1}{q(q-1)}} \,r_j(u^k)^{\frac{1}{q(q-1)}}
\\
&\leq \left(\frac{1}{\theta}\right)^{\frac{1}{q}}\, r_j(u^k)^{\frac{1}{q}}+\left(\frac{L_{u^0}}{\theta} \right)^{\frac{1}{(q-1)}} \left(\frac{1}{\theta}\right)^{\frac{1}{q(q-1)}} \,r_j(u^k)^{\frac{1}{q(q-1)}}.
\end{align*}
Finally, the statement in~\eqref{eq:convofukvk} follows due to~$r_j(u^k)\rightarrow 0$.
\end{proof}
In the following we will use Lemma~\ref{lem:estonukvk} to construct refined sequences~$q_k$ which finally yield improved rates of convergence for~$r_j(u^k)$. By a bootstrapping argument we then also get matching convergence rates for the iterates. We start with the case~$q=2$.
\begin{theorem}
\label{thm:fastconvergence1}
Let Assumption~\ref{ass:fastconv} hold with~$q=2$. Then Algorithm~\ref{alg:gcg} either terminates after~$k \in \N$ iterations with~$u^k$ a minimizer to~\eqref{def:problem} or it generates sequences~$u^k,\, v^k$,~$k\in \N$, with~$u^k \rightharpoonup^* \optu$. In the second case we have
\begin{align*}
r_j(u^k) \leq \lambda^k\, r_j(u^0) \quad \text{where}~\lambda \coloneqq \max \left\{1-\frac{2\alpha \gamma(1-\alpha)}{L_{u^0}\bar{c}^2}, 1-\alpha\right\}
\end{align*}
and~$\bar{c}=c_1+c_2$, see~\eqref{eq:contest2}, and~$\stnorm{u^k-\optu}=\mathcal{O}(\lambda^{k/2})$,~$\stnorm{v^k-\optu}=\mathcal{O}(\lambda^{k/2})$.
\end{theorem}
\begin{proof}
First note that Assumption~\ref{ass:fastconv} subsumes Assumptions~\ref{ass:functions} and~\ref{ass:lipschitz}. Hence Theorem~\ref{thm:slowconvergence} holds and we have~$u^k \rightharpoonup^* \optu$ if Algorithm~\ref{alg:gcg} does not stop after finitely many iterations. Using Lemma~\ref{lem:estsforcontrol}, Theorem \ref{thm:optimality},  and setting~$\bar{c}=c_1+c_2$ we estimate
\begin{align*}
\alpha \, \min \left\{\frac{2\gamma (1-\alpha)\Psi(u^k)}{L_{u^0}\, \stnorm{v_k-u_k}^2}, 1\right\} \geq q_k \coloneqq \alpha \, \min \left\{\frac{2\gamma(1-\alpha)}{L_{u^0}\, \bar{c}}, 1\right\}
\end{align*}
Hence, Proposition~\ref{prop:recursion} yields
\begin{align*}
r_j(u^{k+1}) \leq  \max \left\{1-\frac{2\alpha \gamma(1-\alpha)}{L_{u^0}\bar{c}^2}, 1-\alpha\right\} r_j(u^{k}).
\end{align*}
By induction over~$k$ we arrive at the desired convergence rate. The convergence results for~$u^k$ and~$v^k$, respectively, then follow from~\eqref{eq:contest1}.
\end{proof}
In order to deal with the case~$q>2$ we need yet another recursive lemma.
\begin{lemma}
\label{lem:recursion2}
Let~$h_k \geq 0$,~$k \in \N$, denote a sequence of numbers which satisfy
\begin{align*}
h_{k+1} \leq \max \left\{ \delta,1-C h_k^\beta\right\}\, h_k
\end{align*}
for all~$k \in \N$ and some~$C>0$,~$\delta \in [1/2,1),~0<\beta<1$.
Then we also have
\begin{align*}
h_k \leq \frac{M}{(k+n)^{1/\beta}}
\end{align*}
with constants
\begin{align*}
n\coloneqq\frac{2-(1/\delta)^\beta}{(1/\delta)^\beta-1},~M\coloneqq \max \left\{h_0 n^{1/\beta},\frac{1}{\delta\left(\left(\beta-(1-\beta)(2^\beta-1)C\right)\right)^{1/\beta}}\right\}.
\end{align*}
\end{lemma}
\begin{proof}
This result can be proven analogously to \cite[Theorem 1]{xufrank} where the particular case of~$\delta=1/2$ was treated. We omit a detailed proof at this point.
\end{proof}
With Lemma~\ref{lem:recursion2} at hand, the following improved sublinear rate can be derived.
\begin{theorem}
\label{thm:fastconvergence2}
Let Assumption~\ref{ass:fastconv} hold with~$q>2$. Then Algorithm~\ref{alg:gcg} either terminates after~$k \in \N$ iterations with~$u^k$ a minimizer to~\eqref{def:problem} or it generates sequences~$u^k$,~$v^k$,~$k\in \N$, with~$u^k \rightharpoonup^* \optu$ and~$r_j(u^k)= \mathcal{O}(1/(k+1))$. Set
\begin{align*}
\delta=1-\alpha,~\beta=1-\frac{2}{q(q-1)},~\bar{c}=c_1+c_2,~C= \frac{2\alpha \gamma(1-\alpha)}{L_{u^0}\bar{c}^2}.
\end{align*}
Then there is an index~$K \in \N$ such that
\begin{align*}
r_j(u^k) \leq \frac{M}{(k-K+n)^{1/\beta}}
\end{align*}
for all~$k\geq K$ and
\begin{align*}
n\coloneqq\frac{2-(1/\delta)^\beta}{(1/\delta)^\beta-1},~M\coloneqq \max \left\{r_j(u^K) n^{1/\beta},\frac{1}{\delta\left(\left(\beta-(1-\beta)(2^\beta-1)C\right)\right)^{1/\beta}}\right\}.
\end{align*}
\end{theorem}
\begin{proof}
Again, since Assumption~\ref{ass:fastconv} subsumes Assumptions~\ref{ass:functions} and~\ref{ass:lipschitz}, Theorem~\ref{thm:slowconvergence} holds and we have~$u^k \rightharpoonup^* \optu$ as well as~$r_j(u^k)\rightarrow 0$ with~$r_j(u^k)=\mathcal{O}(1/(k+1))$. In particular there is~$K \in\N$ with~$r_j(u^k)\leq 1$ for all~$k \geq K$. Invoking Lemma~\ref{lem:estsforcontrol} we thus have
\begin{align} \label{eq:dsmall1}
\stnorm{v^k-u^k}\leq \bar{c} \, r_j(u^k)^{\frac{1}{q(q-1)}} \quad \forall k \geq K
\end{align}
where~$\bar{c}=c_1+c_2$.
Similarly to the Theorem~\ref{thm:fastconvergence1} we now estimate
\begin{align*}
\alpha \, \min \left\{\frac{2\gamma(1-\alpha)\Psi(u^k)}{L_{u^0}\, \stnorm{v_k-u_k}^2}, 1\right\} \geq q_k \coloneqq \alpha \, \min \left\{\frac{2\gamma(1-\alpha)}{L_{u^0}\, \bar{c}^2 r_j(u^k)^{\frac{2}{q(q-1)}}}, 1\right\}
\end{align*}
using~\eqref{eq:dsmall1}. Invoking Proposition~\ref{prop:recursion} yields
\begin{align*}
r_j(u^{k+1}) &\leq (1-q_k)\,r_j(u^k)
\\ & = \max \left\{1-\frac{2\alpha \gamma(1-\alpha)}{L_{u^0}\bar{c}^2} r_j(u^{k})^{1-\frac{2}{q(q-1)}}, 1-\alpha\right\}
r_j(u^{k})
\end{align*}
for all~$k\geq K$. The claimed convergence rates are then directly derived from~Lemma~\ref{lem:recursion2}.
\end{proof}
Convergence rates for the iterates~$u^k$ and descent directions~$v^k$ can now be derived according to Lemma~\ref{lem:estonukvk}.
\begin{remark} \label{rem:possiblybetter}
To end this section let us note that the result of Theorem~\ref{thm:fastconvergence2} can be further improved along the lines of~\cite{kerdreux} if there holds
\begin{align} \label{strenghtimpro}
\langle \nabla f(u^k),u^k-v^k  \rangle \geq \theta \, \stnorm{u^k-v^k}^q
\end{align}
for all~$k \in \N$ large enough and some~$\theta>0$. In this case we expect an improved exponent~$\beta= 1/(1-2/q)$. However, in contrast to~\eqref{eq:strenghtfirstorder}, assumptions of the form~\eqref{strenghtimpro} seem to be too restrictive for the applications that we have in mind. For this reason we do not give further details on this topic.
\end{remark}

\section{Applications} \label{sec:applications}
In this final section, the presented algorithm is applied to two model settings. More in detail, we first consider elliptic optimal control problems with an~$L^1$-type regularizer and pointwise control constraints,~\cite{stadler,wachsmuth2,schindele}. These problems are known to promote~\textit{sparse} minimizers i.e. they are zero on large parts of the spatial domain. Second, we turn to a parabolic problem together with a group sparsity penalty similar to~\cite{stadlerparabolic,schneider}. This type of regularization favors~\textit{time-sparse} solutions~$\optu$ which satisfy~$\|\optu(t)\|_{L^2(\Omega)}=0$ on a  subset of the time interval of non-sero measure. For both problems we discuss the computation of the descent direction~$v^k$ and present sufficient structural assumptions implying growth conditions of the form~\eqref{eq:strenghtfirstorder}. The theoretical results are complemented and confirmed by numerical experiments. All computations were carried out in Matlab 2019 on a notebook with~$32$ GB RAM and an Intel\textregistered Core\texttrademark ~i7-10870H CPU@2.20 GHz
\subsection{Bang-bang-off control} \label{subsec:bangbang}
As a first example consider an elliptic control problem of the form
\begin{align} \label{bangbangnonred}
\min_{u \in L^2(\Omega),~y \in L^2(\Omega)} \left \lbrack \frac{1}{2} \|y-y_d\|^2_{L^2(\Omega)}+ \beta \|u\|_{L^1(\Omega)} \right \rbrack
\end{align}
where~$\Omega \subset \R^d$,~$d\in\{1,2,3\}$, is a bounded domain with Lipschitz continuous boundary,~$y_d \in L^2(\Omega)$,~$\beta \geq 0$, and the pair~$(u,y)$ satisfies the state equation
\begin{equation}
- \Delta y  = u+h
~\mbox{in } \Omega ,~y  =  0      ~ \mbox{on } \partial \Omega,
\label{statebangbang}
\end{equation}
~$h \in L^2(\Omega)$, as well as the control constraints
\begin{align*}
u \in \mathcal{U}_{ad} \coloneqq \left\{\,u\in L^\infty(\Omega)\;|\;u_a(x) \leq u(x) \leq u_b(x)\,\right\}
\end{align*}
with
\begin{align*}
u_a,u_b \in L^\infty(\Omega),~u_a(x) \leq 0 \leq u_b(x).
\end{align*}

In order to fit this problem into our setting~\eqref{def:problem}, we set~$\mathcal{U}=\mathcal{C}=L^2(\Omega)$ together with the canonical Hilbert space norm and formally eliminate the state equation.
\begin{proposition} \label{prop:reducedop}
There exists a linear and weak-to-strong continuous selfadjoint operator
\begin{align*}
K \colon L^2(\Omega) \to L^2(\Omega)
\end{align*}
such that~$y=Ku \in H^1_0(\Omega)$ satisfies
\begin{align} \label{eq:stategeneral}
- \Delta y  = u
~\mbox{in } \Omega ,~y  =  0      ~ \mbox{on } \partial \Omega.
\end{align}
in the weak sense.
Moreover there exists~$c>0$ with
\begin{align} \label{eq:l2byl1}
\norm{Ku}_{L^2(\Omega)}\leq c\norm{u}_{L^1(\Omega)}.
\end{align}
\end{proposition}
\begin{proof}
The existence of~$K$ follows immediately from the Lax-Milgram-lemma together with~$H^1_0(\Omega)\hookrightarrow^c L^2(\Omega)$. The estimate in~\eqref{eq:l2byl1} can be proven analogously to~\cite[Proposition 2.8]{pieper15}.
\end{proof}
Consequently, we substitute~$y=Ku$ and consider the reduced problem
\begin{align} \label{def:problembangbangoff}
\min_{u \in\mathcal{U}_{ad}} j(u) \coloneqq \left \lbrack \frac{1}{2} \|Ku-\widehat{y}_d\|^2_{L^2(\Omega)}+ \beta \|u\|_{L^1(\Omega)} \right \rbrack, \tag{$\mathcal{P}_1$}
\end{align}
where~$\widehat{y}_d=y_d-Kh$. In order to prove that this represents a particular instance of~\eqref{def:problem} it will be convenient to introduce the corresponding adjoint equation
\begin{align} \label{eq:adjointgeneral}
- \Delta p  = y-\widehat{y}_d
~\mbox{in } \Omega ,~y  =  0      ~ \mbox{on } \partial \Omega.
\end{align}
\begin{lemma}\label{le5.2}
The functions
\begin{align*}
f(u)=\frac{1}{2} \|Ku-y_d\|^2_{L^2(\Omega)},~g(u)=\beta \|u\|_{L^1(\Omega)}+I_{\mathcal{U}_{ad}}(u)
\end{align*}
satisfy Assumptions~\ref{ass:functions} and~\ref{ass:lipschitz} with
\begin{align*}
\unorm{u}=\norm{u}_{L^2(\Omega)},~\stnorm{u}=\norm{u}_{L^1(\Omega)}
\end{align*}
as well as~$\nabla f(u)=p $ where~$p \in H^1_0(\Omega)$ satisfies~\eqref{eq:adjointgeneral} in the weak sense for~$y=Ku$.
\end{lemma}
\begin{proof}
Assumption~\ref{ass:functions} is readily verified noting that~$K$ is selfadjoint and weak-to-strong continuous, see Proposition~\ref{prop:reducedop}, as well as $\nabla f(u)=K^*(Ku-\widehat{y}_d)$.
Finally every sequence which is bounded in~$L^2(\Omega)$ is also bounded in~$L^1(\Omega)$ and, see~\eqref{eq:l2byl1}, we have
\begin{align*}
(\nabla f(u_1)-\nabla f(u_2),u_3-u_4)_{L^2(\Omega)}&= (K(u_1-u_2),K(u_3-u_4))_{L^2(\Omega)} 
\\ &\leq c^2 \norm{u_1-u_2}_{L^1(\Omega)} \norm{u_3-u_4}_{L^1(\Omega)}
\end{align*}
for all~$u_1,u_2,u_3,u_4 \in L^2(\Omega)$.
\end{proof}
Problems of the form~\eqref{def:problembangbangoff} are appealing for applications such as optimal actuator placement,~\cite{stadler}, since they are known to favor solutions with \textit{bang-bang-off} structure i.e.~$\optu$ is zero on some parts of the domain and achieves the control constraints on the rest of~$\Omega$. These properties can be deduced from the first order optimality conditions.
\begin{proposition} \label{prop:optimalitybangbangoff}
Problem~\eqref{def:problembangbangoff} admits a unique minimizer~$\optu \in \mathcal{U}_{ad}$.
Let~$\bar{p}$ denote the solution of~\eqref{eq:adjointgeneral} for~$y=K\optu$. Then we have
\begin{align*}
\bar{u} \in \begin{cases}
\{u_a\} (x) & \, \bar{p}(x) > \beta \\
\lbrack u_a(x),0 \rbrack & \, \bar{p}(x) = \beta \\
\{u_b(x)\} &\, \bar{p}(x) < -\beta \\
\lbrack 0, u_b(x) \rbrack & \, \bar{p}(x) =- \beta \\
\{0 \} & \,- \beta < \bar{p}(x) < \beta \\
\end{cases}.
\end{align*}
\end{proposition}
\begin{proof}
It is readily verified that the operator~$K$ from Proposition~\ref{prop:reducedop} is injective. Hence~$f$, and thus also~$j$, is strictly convex and the solution to~\eqref{def:problembangbangoff} is unique.
Since~$\optu$ solves~\eqref{def:problembangbangoff} it is also a minimizer to
\begin{align} \label{eq:linearizedbangbang}
\min_{v \in \mathcal{U}_{ad}} \left \lbrack (\bar{p},v)_{L^2(\Omega)}+ \beta \|v\|_{L^1(\Omega)} \right \rbrack= \min_{v \in \mathcal{U}_{ad}} \int_\Omega \bar{p}(x)v(x)+\beta |v(x)|~\mathrm{d}x.
\end{align}
see Theorem~\ref{thm:optimality}.
It is clear that minimizing this integral expression is equivalent to minimizing its integrand in an almost everywhere fashion i.e.~$\optu$ is a solution to~\eqref{eq:linearizedbangbang} if and only if
\begin{align*}
\optu(x) \in \argmin_{\mathbf{v} \in [u_a(x),u_b(x)]} \left \lbrack \bar{p}(x) \mathbf{v}+ \beta |\mathbf{v}| \right \rbrack
\end{align*}
holds for a.e.~$x \in \Omega$.
By a case distinction we arrive at
\begin{align*}
\argmin_{\mathbf{v} \in [u_a(x),u_b(x)]} \left \lbrack \bar{p}(x) \mathbf{v}+ \beta |\mathbf{v}| \right \rbrack= \begin{cases}
\{u_a\} (x) & \, \bar{p}(x) > \beta \\
\lbrack u_a(x),0 \rbrack & \, \bar{p}(x) = \beta \\
\{u_b(x)\} &\, \bar{p}(x) < -\beta \\
\lbrack 0, u_b(x) \rbrack & \, \bar{p}(x) =- \beta \\
\{0 \} & \,- \beta < \bar{p}(x) < \beta \\
\end{cases}
\end{align*}
finishing the proof.
\end{proof}
Hence if we have
\begin{align*}
\operatorname{meas} \left\{\,x \in \Omega:\;|\bar{p}(x)|=\beta \,\right\}=0
\end{align*}
then~$\optu(x)\in \{u_a(x),0, u_b(x)\}$ for a.e.~$x\in \Omega$. In this context, it is also already well-known that conditions of the form~\eqref{eq:strenghtfirstorder} are tightly connected to the behavior of~$\bar{p}$ in the vicinity of~$\{\,x \in \Omega\;|\;|\bar{p}(x)|=\beta\,\}$, see e.g.~\cite{Hinze,wachsmuth2}.
\begin{assumption} \label{ass:bangbang}
There is~$C>0$ and~$\kappa \in (0,1]$ such that
\begin{align*}
\operatorname{meas} \left\{\,x \in \Omega:\;\big |\;|\bar{p}(x)|-\beta\, \big | \leq \eps \,\right\} \leq C \eps^{\kappa}
\end{align*}
for all~$\eps>0$.
\end{assumption}
\begin{proposition}
Let Assumption~\ref{ass:bangbang} hold. Then there exists~$\theta>0$ such that
\begin{align*}
(\bar{p},u-\optu)_{L^2(\Omega)}+ \beta \left( \norm{u}_{L^1(\Omega)}-\norm{\optu}_{L^1(\Omega)} \right) \geq \theta \, \norm{u-\optu}^{1+1/\kappa}_{L^1(\Omega)} \quad \forall u \in \mathcal{U}_{ad}
\end{align*}
i.e. Assumption~\ref{ass:fastconv} is fulfilled with~$q=1+1/\kappa$.
\end{proposition}
\begin{proof}
For a proof see~\cite[Lemma~9.4.2]{Poerner}.
\end{proof}
Finally we address the choice of the Frank-Wolfe descent direction in Algorithm~\ref{alg:gcg}. We observe that~$v^k$ can always be chosen to satisfy the bang-bang-off condition~$v^k(x) \in \{u_a(x),0,u_b(x)\}$ for a.e.~$x\in\Omega$. Moreover it can be characterized  analytically on the basis of  the current adjoint state. As a consequence, the computational cost of step 1. in Algorithm~\ref{alg:gcg} reduces to two elliptic PDE solves for this example.
\begin{lemma}
Let~$u \in \mathcal{U}_{ad}$ be arbitrary and let~$p$ denote the solution of~\eqref{eq:adjointgeneral} with~$y=Ku$. Moreover set
\begin{align*}
\bar{v}= \begin{cases}
u_a (x) & \, p(x) \geq \beta \\
u_b(x) &\, p(x) \leq -\beta \\
0 & \, \text{else}
\end{cases}
\end{align*}
for a.e.~$x\in \Omega$.
Then~$\bar{v} \in \mathcal{U}_{ad}$ is a minimizer of
\begin{align*}
\min_{v \in \mathcal{U}_{ad}} \left \lbrack (p,v)_{L^2(\Omega)}+ \beta \|v\|_{L^1(\Omega)} \right \rbrack.
\end{align*}
\end{lemma}
\begin{proof}
The proof can be carried out along the lines of~Lemma~\ref{prop:optimalitybangbangoff}. 
\end{proof}

The section is finished by two numerical examples. State and adjoint equation are discretized by linear finite elements on a uniform triangulation of~$\Omega=[0,1]^2$ with gridsize~$h=1/256$. The set of admisible controls~$\mathcal{U}_{ad}$ is approximated by piecewise constant functions on the same mesh. For each example, Algorithm~\ref{alg:gcg} is run for a maximum of~$1000$ iterations or until~$\Psi(u^K)\leq 10^{-10}$ for some~$K \in \N$. In the latter case, Theorem~\ref{thm:optimality} also guarantees~$r_j(u^K)\leq 10^{-10}$. Since exact optimal solutions are not available we approximate~$\optu \approx u^K$ for the evaluation of all quantities involving~$\optu$. The Armijo parameters are chosen as~$\alpha=0.5$ and~$\gamma=0.99$.
\begin{example}
\label{bangbang1}[Example 1,~\cite{stadler}]
In the first example we set~$u_a(x) \equiv -30$,~$u_b(x) \equiv 30$ and
\begin{align*}
y_d(x)=\sin (2\pi x_1)\sin(2 \pi x_2) \exp(2x_1)/6,~h(x) \equiv 0,~\beta=0.001.
\end{align*}
\end{example}
\begin{example}[Example 3,~\cite{stadler}]
\label{bangbang2}
In this example we consider spatially varying control constraints and set
\begin{align*}
u_a(x) \equiv -10,~u_b(x)= \begin{cases} 0 & \, (x_1,x_2) \in [0,1/4] \times [0,1]  \\
-5+20x_1 & \, (x_1,x_2) \in [0,1/4] \times [0,1].
\end{cases}
\end{align*}
The desired state is specified to be
\begin{align*}
y_d(x)= \sin(4\pi x_1)\cos(8\pi x_2) \exp (2x_1),~h(x)=10 \cos(8\pi x_1)\sin(8\pi x_2),
\end{align*}
and~$\beta=0.002$. 
\end{example}
\begin{figure}[htb]
\begin{subfigure}[t]{.48\linewidth}
\centering
\includegraphics[scale=0.5]{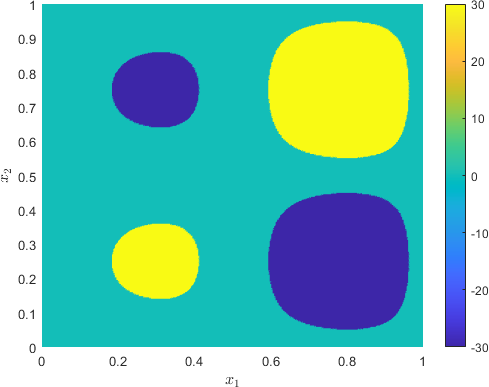}
\caption{Optimal solution~$\optu$.}
\label{fig:solbang1}
\end{subfigure}
\quad
\begin{subfigure}[t]{.48\linewidth}
\centering
\includegraphics[scale=0.5]{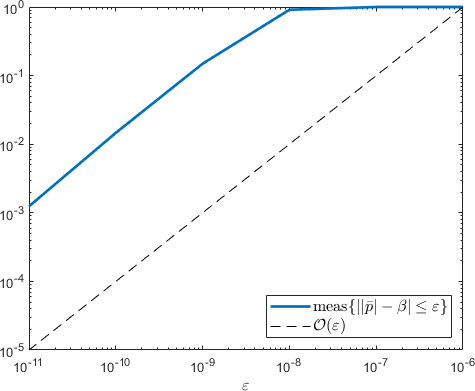}
\caption{Verification of Assumption~\ref{ass:bangbang}.}
\label{fig:measbang1}
\end{subfigure}
\begin{subfigure}[t]{.48\linewidth}
\centering
\includegraphics[scale=0.5]{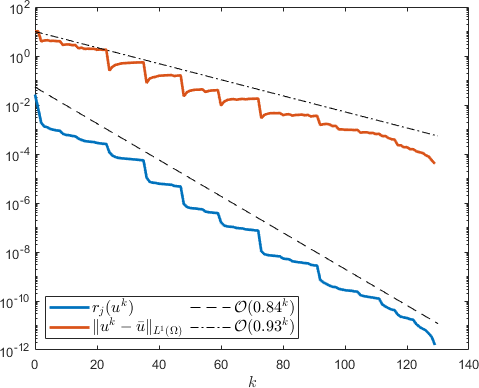}
\caption{Convergence of residuals and iterates.}
\label{fig:convbang1}
\end{subfigure}
\quad
\begin{subfigure}[t]{.48\linewidth}
\centering
\includegraphics[scale=0.5]{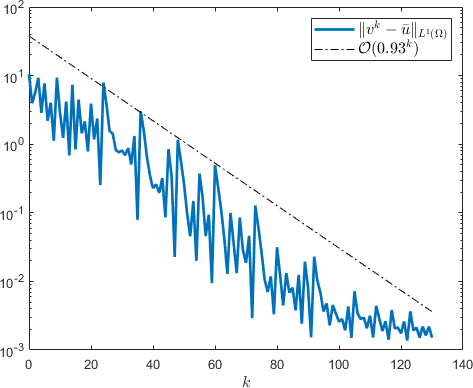}
\caption{Convergence of descent directions.}
\label{fig:convvbang1}
\end{subfigure}
%\quad
%\begin{subfigure}[t]{.31\linewidth}
%\centering
%\includegraphics[scale=0.38]{hessian.pdf}
%\caption{Second derivative~$\bar{p}''$.}
%\label{fig:hessian}
%\end{subfigure}
\caption{Optimal solution~$\optu$ and convergence of relevant quantities for Example~\ref{bangbang1}.}
\label{fig:exbang1}
\end{figure}
\begin{figure}[htb]
\begin{subfigure}[t]{.48\linewidth}
\centering
\includegraphics[scale=0.5]{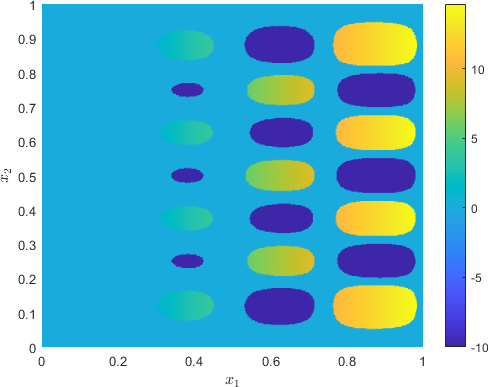}
\caption{Optimal solution~$\optu$.}
\label{fig:solbang2}
\end{subfigure}
\quad
\begin{subfigure}[t]{.48\linewidth}
\centering
\includegraphics[scale=0.5]{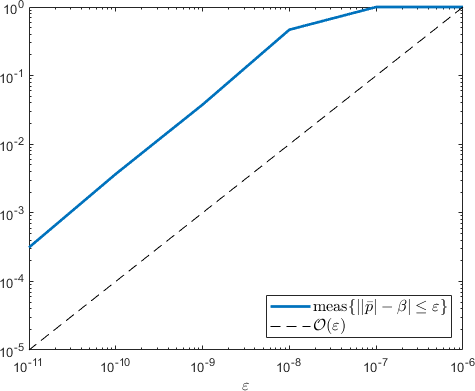}
\caption{Verification of Assumption~\ref{ass:bangbang}.}
\label{fig:measbang2}
\end{subfigure}
\begin{subfigure}[t]{.48\linewidth}
\centering
\includegraphics[scale=0.5]{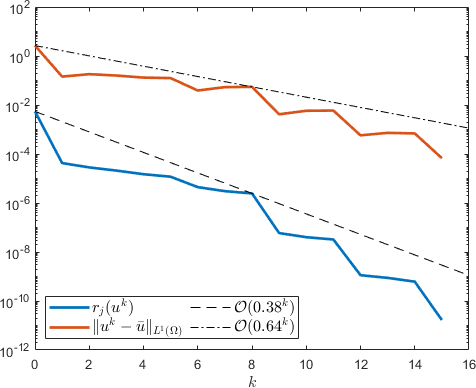}
\caption{Convergence of residuals and iterates.}
\label{fig:convbang2}
\end{subfigure}
\quad
\begin{subfigure}[t]{.48\linewidth}
\centering
\includegraphics[scale=0.5]{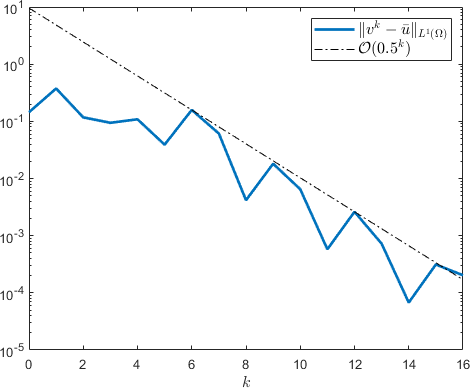}
\caption{Convergence of descent directions.}
\label{fig:convvbang2}
\end{subfigure}
\caption{Optimal solution~$\optu$ and convergence of relevant quantities for Example~\ref{bangbang2}.}
\label{fig:exbang2}
\end{figure}
The results are reported in Figure~\ref{fig:exbang1} and~\ref{fig:exbang2}. For both examples, Algorithm~\ref{alg:gcg} only requires a moderate number of iterations and stops since the termination criterion is fulfilled. The resulting optimal controls are shown in Figure~\ref{fig:solbang1} and~\ref{fig:solbang2}, respectively. In both cases,~$\optu$ satisfies the bang-bang-off condition~$\optu (x) \in \{u_a(x),0,u_b(x)\}$ for almost every~$x \in \Omega$. Moreover, numerical verification, see Figure~\ref{fig:measbang1} and~\ref{fig:measbang2}, suggests that Assumption~\ref{ass:bangbang} is fulfilled with~$\kappa=1$. In this case Theorem~\ref{thm:fastconvergence1} predicts a linear rate of convergence for all relevant quantities. These theoretical results are confirmed by the computations, see Figure~\ref{fig:convbang1} and~\ref{fig:convbang2}, for~$r_j(u^k)$ and~$u^k$, as well as Figure~\ref{fig:convvbang1} and~\ref{fig:convvbang2} for~$v^k$. Finally, while not mentioned in any of the plots, we also point out the  small computational times of~$140$s, for Example~\ref{bangbang1}, and~$11$s, for Example~\ref{bangbang2}, respectively. These are testament to the reasonable per-iteration complexity of the presented method. Together with the reported convergence rates, this underlines the practical utility of Algorithm~\ref{alg:gcg} for this type of problems.
\subsection{Directional sparsity} \label{subsec:direcsparse}
Second, we consider an optimal control problem for the heat equation with sparsity promoting term in the temporal direction. For this purpose, let~$\Omega$ be a bounded domain in $\mathbb{R}^d$,~$d=1,2,3$, with Lipschitz continuous boundary $\partial \Omega$ and~$I=[0,T]$ for some~$T>0$. The concrete minimization problem is given by
\begin{align} \label{eq:direcsparsenonred}
   \Pb \qquad  \min_{u,y}  \frac{1}{2}\int^T_0\|y(t) - y_d(t)\|_{L^2(\Omega)}^2\, dt + \alpha \int^T_0\|u(t)\|_{L^2(\Omega)}\, dt ,
\end{align}
where~$y_d \in L^2(I;L^2(\Omega))$,~$\alpha \geq 0$, and the pair~$(u,y)$ satisfies
\begin{equation}
\left\{
\begin{array}{rcll}
\displaystyle
\partial_t y & = & a \Delta y   + u &
\mbox{in } Q = (0,T) \times \Omega ,\\[1ex]
y & = & 0      & \mbox{on } \Sigma = (0,T)\times \Gamma,\\[1ex]
y(0,\cdot) & = & 0 & \mbox{in } \Omega,
\end{array}
\right.
\label{E1.1}
\end{equation}
~$a >0$, as well as the control constraints
\begin{align*}
u\in \mathcal{U}_{ad} \coloneqq \left\{\, u \in L^2(Q)\;|\;\norm{u(t)}_{L^2(\Omega)}\leq M\, \right\}
\end{align*}
for some~$M>0$. As in Section~\ref{subsec:bangbang},~\eqref{eq:direcsparsenonred} can be cast into a particular instance of the general setting~\eqref{def:problem} by introducing a weak-to-strong continuous and injective operator~$K \colon L^2(Q) \to L^2(Q) $ where~$y=Ku$ satisfies~\eqref{E1.1} in the weak sense. This again leads to a reduced formulation
\begin{align} \label{def:irecsparsered}
\min_{u \in \mathcal{U}_{ad}} j(u)\coloneqq \left \lbrack  \frac{1}{2}\int^T_0\|\lbrack Ku \rbrack (t) - y_d(t)\|_{L^2(\Omega)}^2\, dt + \alpha \int^T_0\|u(t)\|_{L^2(\Omega)}\, dt \right \rbrack \tag{$\mathcal{P}_2$}
\end{align}
%where $T>0, M\in(0,\infty), \Omega$ is a bounded domain in $\mathbb{R}^n$ with Lipschitz continuous boundary $\partial \Omega$ and other normal $\nu$, $y_d\in L^2(0,T;L^2(\Omega))$,  and $B\in L^2(\Omega)$. It is straightforward to argue the existence of a solution $\bar u$ to \Pb. It will be convenient to express $J$ as
%\begin{equation*}
%J(u)= F(u) + \alpha j(u) +I_{U_ad}(u),
%\end{equation*}
%where $F$ denotes the quadratic term in $J$, with $y=y(u)$ the solution to \eqref{E1.1}, further $j(u)= \int^T_0\|u(t)\|_{L^2(\Omega)}\, dt$, and $I_{U_ad}$ is the indicator function of the set $I_{U_ad}$. We shall also utilize the adjoint equation associated to problem \Pb which is given by
in the form of Problem~\eqref{def:problem} with~$\mathcal{U}=L^2(Q)$. In order to characterize the gradient of the tracking-type term it is also convenient to introduce the corresponding adjoint equation
\begin{equation}
\left\{
\begin{array}{rcll}
\displaystyle
-\partial_t p & = & a\Delta p  + y- y_d &
\mbox{in } Q = (0,T) \times \Omega ,\\[1ex]
p & = & 0      & \mbox{on } \Sigma = (0,T)\times \Gamma,\\[1ex]
p(T,\cdot) & = & 0& \mbox{in } \Omega,
\end{array}
\right.
\label{E2.1}
\end{equation}
\begin{lemma}
The functions
\begin{align*}
f(u)=\frac{1}{2}\int^T_0\|\lbrack Ku \rbrack (t) - y_d(t)\|_{L^2(\Omega)}^2\, dt,~g(u)=\alpha \int^T_0\|u(t)\|_{L^2(\Omega)}\, dt  + I_{\mathcal{U}_{ad}}
\end{align*}
satisfy Assumptions~\ref{ass:functions} and~\ref{ass:lipschitz} with
\begin{align*}
\unorm{u}= \norm{u}_{L^2(Q)},~\stnorm{u}=\norm{u}_{L^1(I;L^2(\Omega))}=\int^T_0\|u(t)\|_{L^2(\Omega)}\, dt.
\end{align*}
as well as~$\nabla f=p$ where~$p$ is the weak solution of \eqref{E2.1}  with~$y=Ku$.
\end{lemma}
\begin{proof}
This can be argued along the lines of Lemma \ref{le5.2} by noting that
\begin{align*}
\norm{Ku}_{L^2(Q)} \leq c \norm{u}_{L^1(I;L^2(\Omega))} \quad \forall u \in L^2(Q),
\end{align*}
for a constant $c$ independent of $u\in L^2(Q)$, see~\cite[pg.135, 143]{LSU}, and by characterizing the adjoint operator~$K^*$.
\end{proof}
The regularizer~$g$ is known to encourage~\textit{group-sparse} minimizers which satisfy. As for the bang-bang-off example this can be
\begin{proposition} \label{prop:optimalitydirecsparse}
There exists a unique solution~$\optu \in \mathcal{U}_{ad}$ to~\eqref{def:irecsparsered}. Moreover let~$\bar{p} \in L^2(Q) $ denote the weak solution of~\eqref{E2.1} for~$\bar{y}=K\optu$. Then there holds
\begin{align*}
\bar{u}(t) \in \begin{cases}
\left\{-M \frac{\bar{p}(t)}{\norm{p(t)}}\right\} & \, \norm{\bar{p}(t)} > \alpha \\
\left\{-m \frac{\bar{p}(t)}{\norm{p(t)}}\;|\;m\in[0,M]  \right\} & \, \norm{\bar{p}(t)} = \alpha \\
\{0\} &\, \norm{\bar{p}(t)} < \alpha
\end{cases}
\end{align*}
for a.e.~$t \in I$ where we abbreviate~$\norm{\cdot}=\norm{\cdot}_{L^2(\Omega)}$.
\end{proposition}
\begin{proof}
The proof follows the same steps as Lemma~\ref{prop:optimalitybangbangoff}. First, uniqueness of the minimizer follows due to the strict convexity~$j$. Since~$\optu$ is a minimizer to it also solves
\begin{align*}
\min_{v \in L^2(Q)} \left \lbrack (\bar{p}, v)_{L^2(Q)}+ \alpha~\int^T_0 \norm{v(t)}_{L^2(\Omega)}~\mathrm{d} t\right \rbrack
\end{align*}
This is the case if and only if
\begin{align*}
\optu(t) \in \argmin_{\mathbf{v} \in L^2(\Omega)} \left \lbrack (\bar{p}(t), \mathbf{v})_{L^2(\Omega)}+ \alpha \norm{\mathbf{v}}_{L^2(\Omega)}\right \rbrack
\end{align*}
for a.e.~$t \in I$. The claimed result now follows immediately by a case distinction.
\end{proof}
Consequently we have~$\norm{\optu(t)}_{L^2(\Omega)}=\{0,M\}$ for a.e.~$t \in I$ if~$\|\bar{p}(t)\|_{L^2(\Omega)}=\alpha$ only holds on a zeroset. Similar to the bang-bang-off example, we will now show that strengthened first order conditions,~\eqref{eq:strenghtfirstorder}, for Problem~\eqref{def:irecsparsered} are implied by growth assumptions on~$\norm{\bar{p}(\cdot)}_{L^2(\Omega)}$ in the vicinity of the set
\begin{align*}
\operatorname{meas} \left\{\,t \in I:\;\|\bar{p}(t)\|_{L^2(\Omega)}=\alpha \,\right\}.
\end{align*}
\begin{assumption} \label{ass:direcsparse}
There is~$C>0$ and~$\kappa \in (0,1]$ such that
\begin{align*}
\operatorname{meas} \left\{\,t \in I:\;\big |\;\|\bar{p}(t)\|_{L^2(\Omega)}-\alpha\, \big | \leq \eps \,\right\} \leq C \eps^{\kappa}
\end{align*}
for all~$\eps>0$.
\end{assumption}
We first require the following auxiliary result.
\begin{lemma}
\label{lem:quadgrowthl2}
Let~$p \in L^2(\Omega)$ be given and define~$u=p/\norm{p}_{L^2(\Omega)}$. Then there is~$\sigma>0$ with
\begin{align}
\norm{p}_{L^2(\Omega)}-(p,v)_{L^2(\Omega)} \geq 2\sigma \norm{p}_{L^2(\Omega)} \norm{u-v}^2_{L^2(\Omega)} \quad \forall v,~ \norm{v}_{L^2(\Omega)} \leq 1.
\end{align}
\end{lemma}
\begin{proof}
Since~$L^2(\Omega)$ is a Hilbert space, the norm~$\norm{\cdot}_{L^2(\Omega)}$ is uniformly convex of power type~$2$,  i.e there is~$\sigma >0$ such that for every~$u_1,u_2,u_3 \in L^2(\Omega)$ we have
\begin{align*}
\norm{u_1}_{L^2(\Omega)}, \norm{u_2}_{L^2(\Omega)}\leq 1,~\norm{u_3}_{L^2(\Omega)} \leq \sigma \norm{u_1-u_2}^2_{L^2(\Omega)}  \Rightarrow \left \|(u_1+u_2) /2+u_3 \right\|_{L^2(\Omega)} \leq 1,
\end{align*}
see e.g. \cite{zong}. Now fix~$v \in L^2(\Omega)$,~$\norm{v}_{L^2(\Omega)}\leq 1 $, and let~$z \in L^2(\Omega)$ with~$\norm{z}_{L^2(\Omega)} \leq \sigma \norm{u-v}^2_{L^2(\Omega)}$ be arbitrary. Then using that $(u,p)_{L^2(\Omega)}= \|p\|_{L^2(\Omega)}$ we have
\begin{equation}\label{eq:aux1}
\begin{array}{lll}
\norm{p}_{L^2(\Omega)}&-&({p},v )_{L^2(\Omega)}=( {p}, u-v)_{L^2(\Omega)}\\[1.5ex]
&=& -2\, ( {p}, (u+v)/2 +z )_{L^2(\Omega)}+ 2\, ( {p},z+u )_{L^2(\Omega)} \geq   2\, ( {p}, z)_{L^2(\Omega)}
\end{array}
\end{equation}
where we used~$\norm{(u+v)/2 +z}_{L^2(\Omega)}\leq 1 $ and thus
\begin{align*}
-2\, ( {p}, (u+v)/2 +z )_{L^2(\Omega)} \geq -2 \norm{{p}}_{L^2(\Omega)}=-2\, ( {p}, u )_{L^2(\Omega)}.
\end{align*}
Taking the supremum over all~$z \in L^2(\Omega)$ with~$\norm{z}_{L^2(\Omega)} \leq \sigma \norm{u-v}^2_{L^2(\Omega)}$ on both sides of \eqref{eq:aux1} yields
\begin{align*}
\|p\|_{L^2(\Omega)}- ({p},u )_{L^2(\Omega)} \geq 2\, \sup_{\norm{z}_{L^2(\Omega)}\leq \sigma \norm{u-v}^2_{L^2(\Omega)} } ({p}, z )_{L^2(\Omega)}=2 \sigma \norm{{p}}_{L^2(\Omega)}\, \norm{u-v}^2_{L^2(\Omega)}
\end{align*}
for all $\|v\|_{L^2(\Omega)}\le 1$. This  finishes the proof.
\end{proof}

\begin{proposition} \label{prop:quadgrowthfordirecsparse}
Let Assumption \ref {ass:direcsparse} with~$\kappa \in (0,1]$ hold. Then there exists~$\theta>0$ such that
\begin{align}
(\bar{p},u-\bar{u})_{L^2(I;L^2(\Omega))}+ \alpha \int_I \left( \norm{u(t)}_{L^2(\Omega)}-\norm{\optu(t)}_{L^2(\Omega)} \right)~\mathrm{d}t \geq \theta \norm{u-\bar{u}}^{1+1/\kappa}_{L^1(I;L^2(\Omega))}.
\end{align}
for all~$u \in \mathcal{U}_{ad}$.
\end{proposition}
\begin{proof}
For the proof we can borrow ideas from  \cite[Chapter 9]{Poerner} where a stationary optimal control problem with pointwise, i.e. affine,  control constraints is considered. We
fix an arbitrary~$u \in \mathcal{U}_{ad}$. If not stated otherwise, in the following~$c>0$ denotes a generic constant independent of~$u$. We start by decomposing
\begin{align*}
(\bar{p},u-\bar{u})_{L^2(I;L^2(\Omega))}&+ \alpha \int_I \left( \norm{u(t)}_{L^2(\Omega)}-\norm{\optu(t)}_{L^2(\Omega)} \right)~\mathrm{d}t \\ &=
\int_{\norm{\bar{p}(t)} > \alpha} (\bar{p}(t),u(t)-\bar{u}(t))_{L^2(\Omega)}+ \alpha \left( \norm{u(t)}_{L^2(\Omega)}-\norm{\optu(t)}_{L^2(\Omega)} \right)~\mathrm{d} t\\ &~~~~~+ \int_{\norm{\bar{p}(t)} < \alpha} (\bar{p}(t),u(t))_{L^2(\Omega)}+ \alpha \norm{u(t)}_{L^2(\Omega)}~\mathrm{d}t
\end{align*}
using that ~$\bar{u}(t)=0$ whenever~$\norm{\bar{p}(t)}_{L^2(\Omega)}< \alpha$. For~$\eps>0$ introduce the subsets
\begin{align*}
\mathcal{I}_\eps \coloneqq \left\{\,t \in I\;|\;\norm{\bar{p}(t)} \leq \alpha -\eps\,\right\} \subset \mathcal{I} \coloneqq \left\{\,t \in I\;|\;\norm{\bar{p}(t)} < \alpha \,\right\}.
\end{align*}
Now we estimate
\begin{align*}
\int_{\norm{\bar{p}(t)} < \alpha} (\bar{p}(t),u(t))_{L^2(\Omega)}+ \alpha \norm{u(t)}_{L^2(\Omega)}~\mathrm{d}t &\geq \int_{\mathcal{I}_\eps} \left(\alpha- \norm{\bar{p}(t)}_{L^2(\Omega)} \right)\norm{u(t)}_{L^2(\Omega)}~\mathrm{d}t \\ & \geq
 \eps \norm{u}_{L^1(\mathcal{I}_\eps;L^2(\Omega))} \\ &= \eps \norm{u-\optu}_{L^1(\mathcal{I};L^2(\Omega))} - \eps \norm{u-\optu}_{L^1(\mathcal{I}\setminus \mathcal{I}_\eps;L^2(\Omega))} \\ &\geq \eps \norm{u-\optu}_{L^1(\mathcal{I};L^2(\Omega))} - 2M \eps \operatorname{meas} \left\{\mathcal{I}\setminus \mathcal{I}_\eps\right\}
 \\ & \geq \eps \norm{u-\optu}_{L^1(\mathcal{I};L^2(\Omega))}-\tilde c\eps^{1+\kappa},
\end{align*}
where $\tilde c$ denotes a constant for which without loss of generality we assume $\tilde c \ge 1$. Again we utilized~$\bar{u}(t)=0$ for~$t \in \mathcal{I}$ as well as Assumption~\ref{ass:direcsparse}. Choosing~$\eps= \tilde c^{-\frac{2}{\kappa}}\|u-\bar u\|^{\frac{1}{\kappa}}_{L^1(I;L^2(\Omega))}  $  yields
the existence of a constant $c>0$ such that
\begin{align} \label{eq:diresestpart1}
\int_{\norm{\bar{p}(t)} < \alpha} (\bar{p}(t),u(t))_{L^2(\Omega)}+ \alpha \norm{u(t)}_{L^2(\Omega)}~\mathrm{d}t \geq c \norm{u-\optu}^{1+1/\kappa}_{L^1(\mathcal{I};L^2(\Omega))}.
\end{align}
Similarly define
\begin{align*}
\mathcal{A}_\eps \coloneqq \left\{\,t \in I\;|\;\norm{\bar{p}(t)} \geq \alpha +\eps\,\right\} \subset \mathcal{A} \coloneqq \left\{\,t \in I\;|\;\norm{\bar{p}(t)} > \alpha \,\right\}.
\end{align*}

It remains to estimate from below the expression
\begin{align*}
\int_{\mathcal{A}} (\bar{p}(t)&,u(t)-\bar{u}(t))_{L^2(\Omega)} + \alpha \left( \norm{u(t)}_{L^2(\Omega)}-\norm{\optu(t)}_{L^2(\Omega)} \right)~\mathrm{d} t \\ &= \int_{\mathcal{A}} \left(\norm{\bar{p}(t)}_{L^2(\Omega)}-\alpha\right) | M-\norm{u(t)}_{L^2(\Omega)} |~\mathrm{d}t \\&~~~~~~ +\int_{\mathcal{A}} \left(\bar{p},\norm{u(t)}_{L^2(\Omega)}\left(\bar{p}(t)/\norm{\bar{p}(t)}_{L^2(\Omega)}\right)+u(t)\right)_{L^2(\Omega)}~\mathrm{d} t.
\end{align*}
Here we have used
\begin{align*}
-(\bar{p}(t),\optu(t))_{L^2(\Omega)}=M\norm{\bar{p}(t)}_{L^2(\Omega)}=\norm{\bar{u}(t)}_{L^2(\Omega)} \norm{\bar{p}(t)}_{L^2(\Omega)},~\norm{u(t)}_{L^2(\Omega)} \leq M
\end{align*}
for a.e~$t \in \mathcal{A}$. The first term above is lower bounded by
\begin{align*}
 \int_{\mathcal{A}} \left(\norm{\bar{p}(t)}_{L^2(\Omega)}-\alpha\right) &| M-\norm{u(t)}_{L^2(\Omega)} |~\mathrm{d}t \\&\geq  \int_{\mathcal{A}_\eps} \left(\norm{\bar{p}(t)}_{L^2(\Omega)}-\alpha\right) | M-\norm{u(t)}_{L^2(\Omega)} |~\mathrm{d}t
 \\ & \geq \eps \big\|\;\|u(\cdot)\|_{L^2(\Omega)}-M \big\|_{L^1(\mathcal{A}_\eps)}
 \\ & \geq \eps \big\|\:\|u(\cdot)\|_{L^2(\Omega)}-M\big\|_{L^1(\mathcal{A})}-\eps \big\|\;\|u(\cdot)\|_{L^2(\Omega)}-M \big\|_{L^1(\mathcal{A}\setminus \mathcal{A}_\eps)}.
\end{align*}
As before, by a suitable choice of~$\eps>0 $, we thus arrive at
\begin{align} \label{eq:quadest1help}
\int_{\mathcal{A}} \left(\norm{\bar{p}(t)}_{L^2(\Omega)}-\alpha\right) &|\; \norm{\bar{u}(t)}_{L^2(\Omega)}-\norm{u(t)}_{L^2(\Omega)} |~\mathrm{d}t \geq c \big\| \; \|u(\cdot)\|_{L^2(\Omega)}-M\big\|^{1+1/\kappa}_{L^1(\mathcal{A})}.
\end{align}
Next we show that there holds
\begin{align*}
\left(\bar{p}(t),\norm{u(t)}_{L^2(\Omega)}\left(\bar{p}(t)/\norm{\bar{p}(t)}_{L^2(\Omega)}\right)+u(t)\right)_{L^2(\Omega)} \geq c  \big\| u(t)-\norm{u(t)}_{L^2(\Omega)} \left(\optu(t)/M\right)\big\|^{1+1/\kappa}_{L^2(\Omega)}
\end{align*}
for a.e~$t\in \mathcal{A}$. Of course, this trivially holds if~$u(t)=0$. Otherwise,~if $u(t)\neq 0$, we rewrite
\begin{align*}
(\bar{p}(t),\norm{u(t)}_{L^2(\Omega)}&(\bar{p}(t)/\norm{\bar{p}(t)}_{L^2(\Omega)})+u(t))_{L^2(\Omega)} \\ &=
\norm{u(t)}_{L^2(\Omega)}(\bar{p}(t),\bar{p}(t)/\norm{\bar{p}(t)}_{L^2(\Omega)}+u(t)/\norm{u(t)}_{L^2(\Omega)})_{L^2(\Omega)},
\end{align*}
and apply Lemma~\ref{lem:quadgrowthl2} as well as~$\norm{u(t)}_{L^2(\Omega)}\leq M$,~$\norm{\bar{p}(t)}_{L^2(\Omega)}\geq \alpha$ for a.e~$t\in \mathcal{A}$ to obtain
\begin{align*}
\norm{u(t)}_{L^2(\Omega)}\big(\bar{p}(t)&,\bar{p}(t)/\norm{\bar{p}(t)}_{L^2(\Omega)}+u(t)/\norm{u(t)}_{L^2(\Omega)}\big)_{L^2(\Omega)} \\ &\geq
2\sigma \norm{\bar{p}(t)}_{L^2(\Omega)} \norm{u(t)}_{L^2(\Omega)} \big\|\;\bar{p}(t)/\norm{\bar{p}(t)}_{L^2(\Omega)}+u(t)/\norm{u(t)}_{L^2(\Omega)}\;\big\|^2_{L^2(\Omega)} \\ &\geq
\frac{2\sigma\alpha}{M} \big\|\;\norm{u(t)}_{L^2(\Omega)}(\bar{p}(t)/\norm{\bar{p}(t)}_{L^2(\Omega)})+u(t)\;\big\|^2_{L^2(\Omega)} \\
& \geq c \big\|\;\norm{u(t)}_{L^2(\Omega)}(\bar{p}(t)/\norm{\bar{p}(t)}_{L^2(\Omega)})+u(t)\;\big\|^{1+1/\kappa}_{L^2(\Omega)}.
\end{align*}
Note that, in the last inequality, we have also used~$1/\kappa \geq 1$ and $u\in \mathcal{U}_{ad}$. Thus by integrating over~$\mathcal{A}$ and applying Jensen's integral inequality we get
\begin{align} \label{eq:eq:quadest2help}
\int_{\mathcal{A}} (\bar{p},&\norm{u(t)}_{L^2(\Omega)}\big(\bar{p}(t)/\norm{\bar{p}(t)}_{L^2(\Omega)})+u(t)\big)_{L^2(\Omega)}~\mathrm{d} t \nonumber \\
&\geq c \int_{\mathcal{A}} \big\|\, u(t)-\norm{u(t)}_{L^2(\Omega)} \left(\optu(t)/M\right)\big\|^{1+1/\kappa}_{L^2(\Omega)}~\mathrm{d}t
\nonumber \\
&\geq \frac{c}{T^{1/\kappa}} \left( \int_{\mathcal{A}} \big\| \, u(t)-\norm{u(t)}_{L^2(\Omega)} \left(\optu(t)/M\right)\big\|_{L^2(\Omega)}~\mathrm{d}t \right)^{1+1/\kappa} \nonumber \\
& =\frac{c}{T^{1/\kappa}} \big\|\, u-\norm{u(\cdot)}_{L^2(\Omega)} \left(\optu/M\right)\big\|^{1+1/\kappa}_{L^1(\mathcal{A};L^2(\Omega))}
\end{align}
Combining estimates~\eqref{eq:quadest1help} and~\eqref{eq:eq:quadest2help}, and applying Jensen's inequality as well as
\begin{align*}
\norm{u-\bar{u}}_{L^1(\mathcal{A};L^2(\Omega))}
&\leq \big\|\,u-\norm{u(\cdot)}_{L^2(\Omega)} (\optu/M)\, \big\|_{L^1(\mathcal{A};L^2(\Omega))} + \big\|\, \norm{u(\cdot)}_{L^2(\Omega)} (\optu/M) - M (\bar u/M)\big\|_{L^1(\mathcal{A})} \\
&\leq \big\|\,u-\norm{u(\cdot)}_{L^2(\Omega)} (\optu/M)\, \big\|_{L^1(\mathcal{A};L^2(\Omega))} + \big\|\, \norm{u(\cdot)}_{L^2(\Omega)}-M \big\|_{L^1(\mathcal{A})}
\end{align*}
yields
\begin{align} \label{eq:diresestpart2}
\int_{\mathcal{A}} (\bar{p}(t)&,u(t)-\bar{u}(t))_{L^2(\Omega)} + \alpha ( \norm{u(t)}_{L^2(\Omega)}-\norm{\optu(t)}_{L^2(\Omega)} )~\mathrm{d} t \nonumber \\
&\geq c \left( \big\|\, \norm{u(\cdot)}_{L^2(\Omega)}-M\,\big\|^{1+1/\kappa}_{L^1(\mathcal{A})}+
\big\|\,u(t)-\norm{u(t)}_{L^2(\Omega)} (\optu(t)/M)\,\big\|^{1+1/\kappa}_{L^1(\mathcal{A};L^2(\Omega))} \right) \nonumber
\\&\geq c \norm{u-\optu}^{1+1/\kappa}_{L^1(\mathcal{A};L^2(\Omega))}.
\end{align}
Finally, combining~\eqref{eq:diresestpart1},~\eqref{eq:diresestpart2}, noting that~$I \setminus \{\mathcal{I} \cup \mathcal{A}\}$ is a set of measure, and another application of Jensen's inequality, we get
\begin{align*}
(\bar{p},u-\bar{u})_{L^2(I;L^2(\Omega))}&+ \alpha \int_I \left( \norm{u(t)}_{L^2(\Omega)}-\norm{\optu(t)}_{L^2(\Omega)} \right)~\mathrm{d}t \\
& \geq c \left( \norm{u-\optu}^{1+1/\kappa}_{L^1(\mathcal{I};L^2(\Omega))}+\norm{u-\optu}^{1+1/\kappa}_{L^1(\mathcal{A};L^2(\Omega))} \right) \\ & \geq
\theta \norm{u-\optu}^{1+1/\kappa}_{L^1(\mathcal{I} \cup\mathcal{A};L^2(\Omega))}
= \theta \norm{u-\optu}^{1+1/\kappa}_{L^1(I;L^2(\Omega))}
\end{align*}
for some~$\theta>0$. This completes the proof.
\end{proof}
To complete the discussion of~Problem~\eqref{def:irecsparsered}, we note that the GCG descent direction~$v^k$ can be given explicitly once the current adjoint state is computed.
\begin{lemma}
Let~$u \in \mathcal{U}_{ad}$ be arbitrary and let~$p$ denote the solution of~\eqref{E2.1} with~$y=Ku$. Abbreviating~$\norm{\cdot}=\norm{\cdot}_{L^2(\Omega)}$ set
\begin{align*}
\bar{v}= \begin{cases}
-M \frac{p(t)}{\norm{p(t)}} & \, \norm{p(t)} \geq \alpha \\
0 & \, \text{else}
\end{cases}
\end{align*}
for a.e.~$t\in I$.
Then~$\bar{v} \in \mathcal{U}_{ad}$ is a minimizer of
\begin{align*}
\min_{v \in \mathcal{U}_{ad}} \left \lbrack (p,v)_{L^2(Q)}+ \alpha \int^T_0 \|v(t)\|_{L^2(\Omega)}~\mathrm{d}t \right \rbrack.
\end{align*}
\end{lemma}
\begin{proof}
This can be proven with the same arguments as in Proposition~\ref{prop:optimalitydirecsparse}.
\end{proof}
Finally we demonstrate Algorithm~\ref{alg:gcg} on the following setting:
\begin{example} \label{ex:direcsparsity}
Let~$\Omega=[0,1]$ and~$T=1$. The conductivity~$a$ is chosen as~$a=0.7$. Moreover we set
\begin{align*}
y_d(x,t)=\sin (2\pi x_1)\sin(2 \pi x_2)\sin(\pi t) \exp(2x_1)/6,~\beta=0.0035,~M=0.8.
\end{align*}
\end{example}

The state equation is discretized by an implicit Euler method with stepsize~$\tau=1/500$ in time and by linear finite elements on a uniform triangulation of~$\Omega$ with gridsize~$h=1/64$. The set of admissible constraints~$\mathcal{U}_{ad}$ is also discretized using continuous piecewise linear function on the same grid. The adjoint equation is discretized consistently. As before, Algorithm~\ref{alg:gcg} is run for a maximum of~$1000$ iterations or until~$\Psi(u^K)\leq 10^{-10}$ for some~$k \in\N$. In the second case we use~$\optu \approx u^K $ in order to compute the residual and all other relevant quantities. The Quasi-Armijo parameters are again chosen as~$\alpha=0.5$ and~$\gamma=0.99$.
\begin{figure}[htb]
\begin{subfigure}[t]{.48\linewidth}
\centering
\includegraphics[scale=0.5]{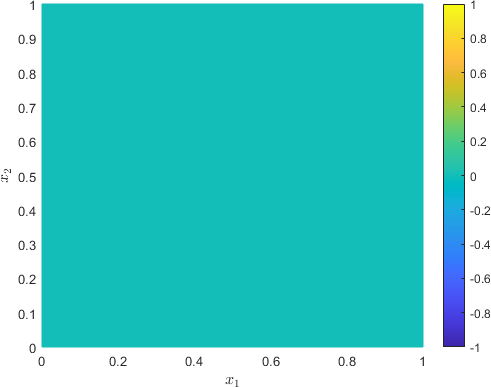}
\caption{Optimal control~$\optu(t)$,~$t=0.2$.}
\label{fig:sol}
\end{subfigure}
\quad
\begin{subfigure}[t]{.48\linewidth}
\centering
\includegraphics[scale=0.5]{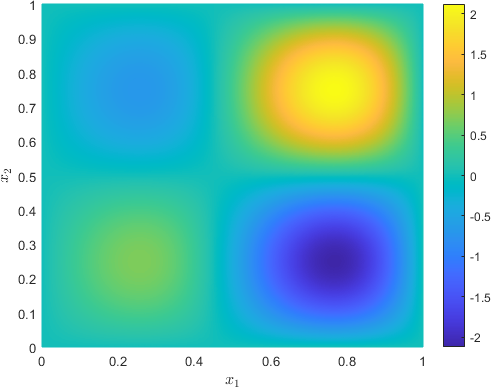}
\caption{Optimal control~$\optu(t)$,~$t=0.4$.}
\label{fig:adjointval}
\end{subfigure}
\begin{subfigure}[t]{.48\linewidth}
\centering
\includegraphics[scale=0.5]{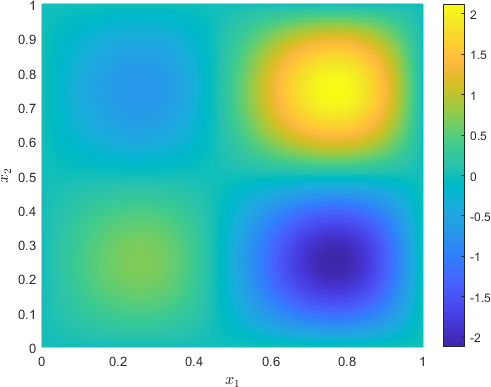}
\caption{Optimal control~$\optu(t)$,~$t=0.6$.}
\label{fig:sol}
\end{subfigure}
\quad
\begin{subfigure}[t]{.48\linewidth}
\centering
\includegraphics[scale=0.5]{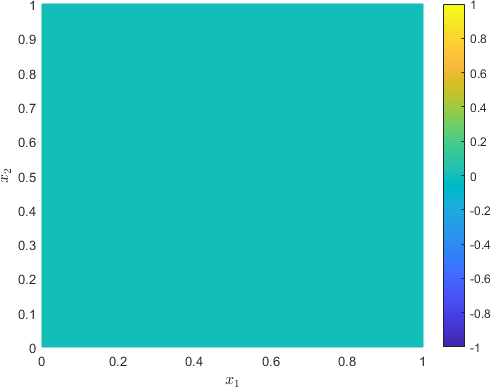}
\caption{Optimal control~$\optu(t)$,~$t=0.8$.}
\label{fig:adjointval}
\end{subfigure}
%\quad
%\begin{subfigure}[t]{.31\linewidth}
%\centering
%\includegraphics[scale=0.38]{hessian.pdf}
%\caption{Second derivative~$\bar{p}''$.}
%\label{fig:hessian}
%\end{subfigure}
\caption{Snapshots of the optimal control~$\optu$ at several time instances.}
\label{fig:optcontrolparab}
\end{figure}
\begin{figure}[htb]
\begin{subfigure}[t]{.48\linewidth}
\centering
\includegraphics[scale=0.5]{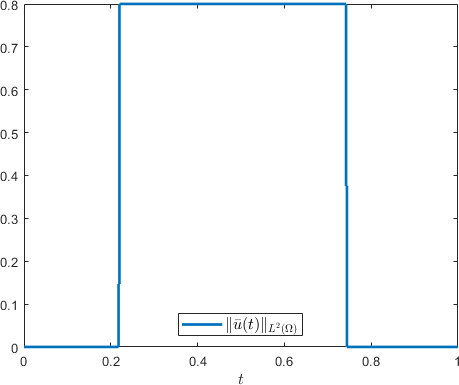}
\caption{$L^2(\Omega)$-norm of~$\optu(t)$.}
\end{subfigure}
\quad
\begin{subfigure}[t]{.48\linewidth}
\centering
\includegraphics[scale=0.5]{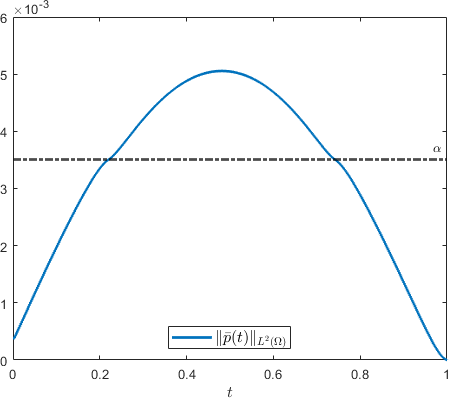}
\caption{$L^2(\Omega)$-norm of~$\bar{p}(t)$.}
\label{fig:normspara}
\end{subfigure}
\caption{Evolution of~$\norm{\bar{u}(t)}_{L^2(\Omega)}$ and $\norm{\bar{p}(t)}_{L^2(\Omega)}$.}
\label{fig:sparsitydirec}
\end{figure}
\begin{figure}[htb]
\begin{subfigure}[t]{.31\linewidth}
\centering
\includegraphics[scale=0.38]{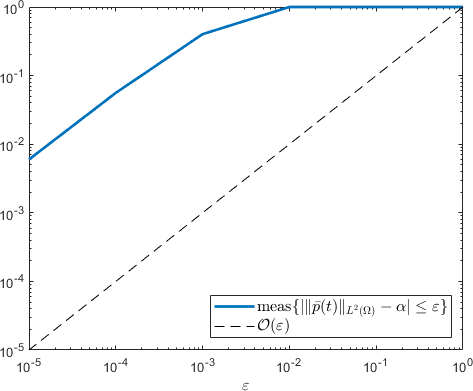}
\caption{Assumption~\ref{ass:direcsparse}.}
\label{fig:measpara}
\end{subfigure}
\quad
\begin{subfigure}[t]{.31\linewidth}
\centering
\includegraphics[scale=0.38]{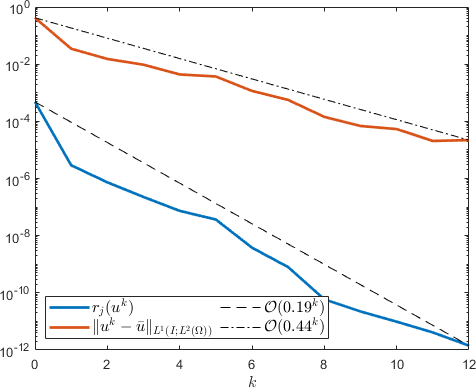}
\caption{Residuals and iterates.}
\label{fig:resanditspara}
\end{subfigure}
\quad
\begin{subfigure}[t]{.31\linewidth}
\centering
\includegraphics[scale=0.38]{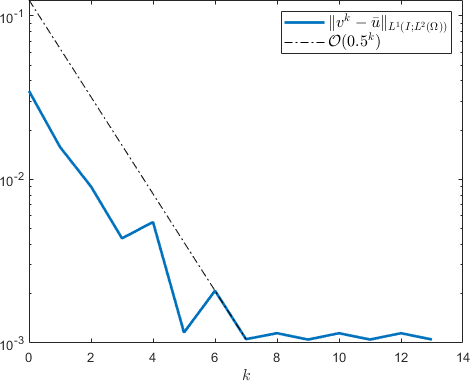}
\caption{Descent directions.}
\label{fig:descentparab}
\end{subfigure}
\caption{Verification of Assumption~\ref{ass:direcsparse} and convergence of relevant quantities.}
\label{fig:convpara}
\end{figure}

The computed results can be found in Figure~\ref{fig:optcontrolparab},~\ref{fig:sparsitydirec} and~\ref{fig:convpara}, respectively. Most notably,~Algorithm~\ref{alg:gcg} terminates after mere~$13$ iterations. Snapshots of the optimal control~$\optu$ at~$t=0.2,0.4,0.6,0.8$, are plotted in Figure~\ref{fig:optcontrolparab}. In order to assess the sparsity pattern of the minimizer, the~$L^2(\Omega)$-norms of~$\optu(t)$ and~$\bar{p}(t)$ are shown in Figure~\ref{fig:normspara}. As predicted by Proposition~\ref{prop:optimalitydirecsparse} there holds~$\bar{u}(t)=0$ if~$\norm{{\bar{p}(t)}}_{L^2(\Omega)}<\alpha$. Moreover the set
\begin{align*}
\left\{\,t \in I\;|\;\norm{\bar{p}(t)}=\alpha\,\right\}
\end{align*}
is a zeroset and thus~$\norm{\optu(t)}_{L^2(\Omega)} \in\{0,M\}$ for a.e.~$t\in I$.
Next we numerically verify the growth condition from Assumption~\ref{ass:direcsparse} with~$\kappa=1$ in Figure~\ref{fig:measpara}. Accordingly, Theorem~\ref{thm:fastconvergence1} predicts a linear rate of convergence for all relevant quantities. Indeed, this is clearly observed for the residuals~$r_j(u^k)$ as well as the iterates~$u^k$, see~Figure~\ref{fig:resanditspara}. For~$v^k$ a similar behavior can be deduced from Figure~\ref{fig:descentparab}, at least in the earlier iterations. However, for~$k \geq 7$, the error stops decreasing. We suppose that this is a numerical artifact which is caused by the approximation of~$\optu$ by~$u^K$. Finally, the overall computational time amounts to around $300$s. Altogether, these observations confirm our theoretical results and further highlight the practical utility of GCG methods for nonsmooth minimization.
\bibliographystyle{siam}
\bibliography{references}

\begin{thebibliography}{10}

\bibitem{metricsub}
{\sc F.~J.~A. {Artacho} and M.~H. {Geoffroy}}, {\em {Metric subregularity of
  the convex subdifferential in Banach spaces}}, {J. Nonlinear Convex Anal.},
  15 (2014), pp.~35--47.

\bibitem{beckbook}
{\sc A.~{Beck}}, {\em {First-order methods in optimization}}, vol.~25,
  Philadelphia, PA: Society for Industrial and Applied Mathematics (SIAM);
  Philadelphia, PA: Mathematical Optimization Society (MOS), 2017.

\bibitem{beck}
{\sc A.~{Beck}, E.~{Pauwels}, and S.~{Sabach}}, {\em {The cyclic block
  conditional gradient method for convex optimization problems}}, {SIAM J.
  Optim.}, 25 (2015), pp.~2024--2049.

\bibitem{bertsekas}
{\sc D.~P. {Bertsekas}}, {\em {Nonlinear programming}}, Belmont, MA: Athena
  Scientific, 2016.

\bibitem{fanzon}
{\sc K.~Bredies, M.~Carioni, S.~Fanzon, and F.~Romero}, {\em A generalized
  conditional gradient method for dynamic inverse problems with optimal
  transport regularization}, 2020.

\bibitem{bredies}
{\sc K.~{Bredies}, D.~A. {Lorenz}, and P.~{Maass}}, {\em {A generalized
  conditional gradient method and its connection to an iterative shrinkage
  method}}, {Comput. Optim. Appl.}, 42 (2009), pp.~173--193.

\bibitem{cannon}
{\sc M.~D. {Canon} and C.~D. {Cullum}}, {\em {A tight upper bound on the rate
  of convergence of the Frank-Wolfe algorithm}}, {SIAM J. Control}, 6 (1968),
  pp.~509--516.

\bibitem{pokuttacomplex}
{\sc C.~W. Combettes and S.~Pokutta}, {\em Complexity of linear minimization
  and projection on some sets}, 2021.
\newblock \url{https://arxiv.org/abs/2101.10040}.

\bibitem{Hinze}
{\sc K.~{Deckelnick} and M.~{Hinze}}, {\em {A note on the approximation of
  elliptic control problems with bang-bang controls}}, {Comput. Optim. Appl.},
  51 (2012), pp.~931--939.

\bibitem{demyanov}
{\sc V.~F. {Dem'yanov} and A.~M. {Rubinov}}, {\em {Approximate methods in
  optimization problems}}.
\newblock {Translated from the Russian by Scripta Technica, Inc. Translation
  editor: George M. Kranc. (Modern Analytic and Computational Methods in
  Science and Mathematics. No. 32.) New York: American Elsevier Publishing
  Company, Inc. 1970. IX, 256 p., Dfl. 80.00. (1970).}, 1970.

\bibitem{denoyelle}
{\sc Q.~{Denoyelle}, V.~{Duval}, G.~{Peyr\'e}, and E.~{Soubies}}, {\em {The
  sliding Frank-Wolfe algorithm and its application to super-resolution
  microscopy}}, {Inverse Probl.}, 36 (2020), p.~42.
\newblock Id/No 014001.

\bibitem{Dunn79}
{\sc J.~C. Dunn}, {\em Rates of convergence for conditional gradient algorithms
  near singular and nonsingular extremals}, SIAM J. Control Optim., 17 (1979),
  pp.~187--211.

\bibitem{dunnimplicit}
{\sc J.~C. {Dunn}}, {\em {Convergence rates for conditional gradient sequences
  generated by implicit step length rules}}, {SIAM J. Control Optim.}, 18
  (1980), pp.~473--487.

\bibitem{dunnopen}
{\sc J.~C. {Dunn} and S.~{Harshbarger}}, {\em {Conditional gradient algorithms
  with open loop step size rules}}, {J. Math. Anal. Appl.}, 62 (1978),
  pp.~432--444.

\bibitem{ekeland}
{\sc I.~{Ekeland} and R.~{T\'emam}}, {\em {Convex analysis and variational
  problems.}}, vol.~28, Philadelphia, PA: Society for Industrial and Applied
  Mathematics, 1999.

\bibitem{flinth}
{\sc A.~Flinth, F.~de~Gournay, and P.~Weiss}, {\em On the linear convergence
  rates of exchange and continuous methods for total variation minimization},
  Mathematical Programming,  (2020).

\bibitem{frank}
{\sc M.~Frank and P.~Wolfe}, {\em An algorithm for quadratic programming},
  Naval Res. Logist. Quart., 3 (1956), pp.~95--110.

\bibitem{garber}
{\sc D.~Garber and E.~Hazan}, {\em Faster rates for the frank-wolfe method over
  strongly-convex sets}, in Proceedings of the 32nd International Conference on
  Machine Learning, {ICML} 2015, Lille, France, 6-11 July 2015, F.~R. Bach and
  D.~M. Blei, eds., vol.~37 of {JMLR} Workshop and Conference Proceedings,
  JMLR.org, 2015, pp.~541--549.

\bibitem{guelat}
{\sc J.~{Gu\'elat} and P.~{Marcotte}}, {\em {Some comments on Wolfe's 'away
  step'}}, {Math. Program.}, 35 (1986), pp.~110--119.

\bibitem{guillerme}
{\sc J.~{Guillerme}}, {\em {Intermediate value theorems and fixed point
  theorems for semi-continuous functions in product spaces}}, {Proc. Am. Math.
  Soc.}, 123 (1995), pp.~2119--2122.

\bibitem{stadlerparabolic}
{\sc R.~{Herzog}, G.~{Stadler}, and G.~{Wachsmuth}}, {\em {Directional sparsity
  in optimal control of partial differential equations}}, {SIAM J. Control
  Optim.}, 50 (2012), pp.~943--963.

\bibitem{jaggi}
{\sc M.~Jaggi}, {\em Revisiting {Frank-Wolfe}: Projection-free sparse convex
  optimization}, in Proceedings of the 30th International Conference on Machine
  Learning, S.~Dasgupta and D.~McAllester, eds., no.~1 in Proceedings of
  Machine Learning Research, Atlanta, Georgia, USA, 17--19 Jun 2013, PMLR,
  pp.~427--435.

\bibitem{kerdreux}
{\sc T.~Kerdreux, A.~d'Aspremont, and S.~Pokutta}, {\em Projection-free
  optimization on uniformly convex sets}, in Proceedings of The 24th
  International Conference on Artificial Intelligence and Statistics,
  A.~Banerjee and K.~Fukumizu, eds., vol.~130 of Proceedings of Machine
  Learning Research, PMLR, 13--15 Apr 2021, pp.~19--27.

\bibitem{LSU}
{\sc O.~A. Lady\v{z}enskaja, V.~A. Solonnikov, and N.~N. Uralceva}, {\em Linear
  and quasilinear equations of parabolic type}, Translations of Mathematical
  Monographs, Vol. 23, American Mathematical Society, Providence, R.I., 1968.
\newblock Translated from the Russian by S. Smith.

\bibitem{levitin}
{\sc E.~Levitin and B.~Polyak}, {\em Constrained minimization methods}, USSR
  Computational Mathematics and Mathematical Physics, 6 (1966), pp.~1 -- 50.

\bibitem{neitzel}
{\sc I.~{Neitzel}, K.~{Pieper}, B.~{Vexler}, and D.~{Walter}}, {\em {A sparse
  control approach to optimal sensor placement in PDE-constrained parameter
  estimation problems}}, {Numer. Math.}, 143 (2019), pp.~943--984.

\bibitem{pieper15}
{\sc K.~Pieper}, {\em Finite element discretization and efficient numerical
  solution of elliptic and parabolic sparse control problems}, {PhD}
  {D}issertation, Technische Universit{\"a}t M{\"u}nchen, 2015.
\newblock
  \url{http://nbn-resolving.de/urn/resolver.pl?urn:nbn:de:bvb:91-diss-20150420-1241413-1-4
  }.

\bibitem{tang}
{\sc K.~{Pieper}, B.~Q. {Tang}, P.~{Trautmann}, and D.~{Walter}}, {\em {Inverse
  point source location with the Helmholtz equation on a bounded domain}},
  {Comput. Optim. Appl.}, 77 (2020), pp.~213--249.

\bibitem{pieper}
{\sc K.~{Pieper} and D.~{Walter}}, {\em {Linear convergence of accelerated
  conditional gradient algorithms in spaces of measures}}, {ESAIM, Control
  Optim. Calc. Var.}, 27 (2021), p.~37.
\newblock Id/No 38.

\bibitem{Poerner}
{\sc F.~P\"orner}, {\em Regularization Methods for Ill-Posed Optimal Control
  Problems}, dissertation, Universit\"at W\"urzburg, 2018.
\newblock
  \url{http://nbn-resolving.de/urn/resolver.pl?urn:nbn:de:bvb:20-opus-163153 }.

\bibitem{rakotomamonjy}
{\sc A.~Rakotomamonjy, R.~Flamary, and N.~Courty}, {\em Generalized conditional
  gradient: analysis of convergence and applications}, 2015.
\newblock https://arxiv.org/abs/1510.06567.

\bibitem{rockafellar}
{\sc R.~T. {Rockafellar}}, {\em {Convex analysis}}, Princeton, NJ: Princeton
  University Press, 1997.

\bibitem{schindele}
{\sc A.~Schindele and A.~Borz{\`\i}}, {\em Proximal methods for elliptic
  optimal control problems with sparsity cost functional}, Applied Mathematics,
  7 (2016), p.~967.

\bibitem{schneider}
{\sc C.~{Schneider} and G.~{Wachsmuth}}, {\em {Regularization and
  discretization error estimates for optimal control of ODEs with group
  sparsity}}, {ESAIM, Control Optim. Calc. Var.}, 24 (2018), pp.~811--834.

\bibitem{stadler}
{\sc G.~{Stadler}}, {\em {Elliptic optimal control problems with
  \(L^1\)-control cost and applications for the placement of control devices}},
  {Comput. Optim. Appl.}, 44 (2009), pp.~159--181.

\bibitem{trautmann}
{\sc P.~Trautmann and D.~Walter}, {\em A fast primal-dual-active-jump method
  for minimization in $\operatorname{BV}((0,t);\mathbb{R}^d)$}, 2021.
\newblock \url{https://arxiv.org/abs/2106.00633}.

\bibitem{wachsmuth2}
{\sc G.~{Wachsmuth} and D.~{Wachsmuth}}, {\em {Convergence and regularization
  results for optimal control problems with sparsity functional}}, {ESAIM,
  Control Optim. Calc. Var.}, 17 (2011), pp.~858--886.

\bibitem{xu2017}
{\sc H.-K. Xu}, {\em Convergence analysis of the frank-wolfe algorithm and its
  generalization in banach spaces}, 2017.
\newblock \url{https://arxiv.org/abs/1710.07367 }.

\bibitem{xufrank}
{\sc Y.~Xu and T.~Yang}, {\em Frank-wolfe method is automatically adaptive to
  error bound condition}, 2018.
\newblock \url{https://arxiv.org/abs/1810.04765}.

\bibitem{zong}
{\sc Z.-B. {Xu} and G.~F. {Roach}}, {\em {Characteristic inequalities of
  uniformly convex and uniformly smooth Banach spaces}}, {J. Math. Anal.
  Appl.}, 157 (1991), pp.~189--210.

\bibitem{yu}
{\sc Y.~{Yu}, X.~{Zhang}, and D.~{Schuurmans}}, {\em {Generalized conditional
  gradient for sparse estimation}}, {J. Mach. Learn. Res.}, 18 (2017), p.~46.
\newblock Id/No 144.

\end{thebibliography}
\end{document}